\documentclass[a4paper,reqno,11pt]{amsart}
\usepackage{amsmath}
\usepackage{amsthm}
\usepackage{amscd}
\usepackage{amssymb}
\usepackage{amsfonts}
\usepackage{latexsym}
\usepackage[mathscr]{eucal}

\newtheorem{thm}{Theorem}[section]
\newtheorem{prop}[thm]{Proposition}
\newtheorem{lemma}[thm]{Lemma}
\newtheorem{cor}[thm]{Corollary}
\newtheorem{defn}[thm]{Definition}
\newtheorem{exam}{Example}[section]

\def\O{{\mathcal O}}

\def\Aut{\hbox{\rm Aut}}
\def\K{{\mathcal K}}

\def\T{{\mathcal T}}

\def\F{{\mathcal F}}
\def\L{{\mathcal L}}

\def\C{{\mathcal C}}

\def\comp{{\mathbf {C}}}

\def\cst{{C${}^*$}}

\begin{document}
\title{Traces on cores of \cst -algebras associated with self-similar maps}
\date{}
\author{Tsuyoshi Kajiwara}
\address[Tsuyoshi Kajiwara]{Department of Environmental and
Mathematical Sciences,
Okayama University, Tsushima, 700-8530,  Japan}

\author{Yasuo Watatani}
\address[Yasuo Watatani]{Department of Mathematical Sciences,
Kyushu University, Motooka, Fukuoka, 819-0395, Japan}
\maketitle
 \hyphenation{cor-re-spond-ences}
\begin{abstract}
We completely classify  the extreme tracial states on
the cores of the \cst -algebras associated with  self-similar maps on
compact metric spaces.  We present a complete list of them. 
The extreme tracial states are the 
union of the 
discrete type tracial states given by measures supported on the 
finite  orbits of the branch
points and a continuous type  tracial state given by the Hutchinson measure
on the original self-similar set. 

\medskip\par\noindent
AMS SUBJECT CLASSIFICATION: 46L08, 46L55

\end{abstract}
\section
{Introduction}
We investigate a relation between self-similar sets and $C^*$-algebras. 
Most of self-similar sets are constructed from iterations of proper
contractions.   A self-similar map on a compact metric
space $K$ is a family of proper contractions 
$\gamma = (\gamma_1,\dots,\gamma_N)$ on $K$ such that $K =
\bigcup_{i=1}^N\gamma_i(K)$.  In our former work 
Kajiwara-Watatani \cite{KW1}, we introduced $C^*$-algebras 
associated with self-similar maps on compact metric spaces as 
Cuntz-Pimsner algebras for certain \cst-correspondences
and show that the associated \cst
-algebras are simple and purely infinite.  A related study 
on $C^*$-algebras associated with iterated function systems 
is  done by Castro \cite{Ca}.  
A generalization to 
Mauldin-Williams graphs is given by Ionescu-Watatani \cite{IW}. 

The gauge action of the associated \cst-algebra is an 
 important tool. The fixed point algebra under the gauge action is 
called the core. Recall that the dimension group of a topological Markov shift 
is exactly the K-theory of the core 
of the corresponding Cuntz-Krieger algebra \cite{CK}. 
This idea is extended to the subshifts
 by Matsumoto \cite{Ma}.  Therefore we expect that many 
informations of  a self-similar map  
as a dynamical system are contained in the structure of the  core of the 
corresponding $C^*$-algebra. In this note 
we study the trace structure of the fixed point algebra under the gauge action. In fact, the trace structure is described  by  measures supported on the 
orbits of the 
branched points of the  self-similar map and the Hutchinson measure. 
We present a complete list of the extreme traces on the core to classify them.

One of the key points is the structure of the cores of
Cuntz-Pimsner algebras described by Pimsner \cite{Pi}.  
We need a result on the extension of traces on a subalgebra and
an ideal to their sum modified in \cite{KW2}  following after Exel and Laca 
\cite{EL}. We also recall the Rieffel
correspondence of traces between Morita equivalent \cst -algebras, 
which plays an important role in our study.  

The extreme traces are described by  measures supported on the 
finite orbits of branched points 
of a self-similar map. 
For a branched point $b$ and an integer $r \geq 0$, we consider a family 
$(\mu^{(b,r)}_i)_{i= 0}^r$ of 
discrete measures on $K$ satisfying a compatibility condition such that 
the support of  $\mu^{(b,r)}_i$ is on the 
$(r-i)$-th $\gamma$-orbit of the branch point $b$ for $0 \leq i \leq r$. 
 By the Rieffel correspondence
of traces, we construct finite  traces  on the compact
algebras of multiple tensor products of the \cst -correspondence constructed
from the original self-similar map.  A sequence of such compatible 
traces gives an
extreme trace $\tau^{(b,r)}$ on the core.  
On the other hand, the Hutchinson measure on the
self-similar set $K$ also gives a continuous type trace $\tau^{\infty}$ 
 on the core.
We shall show that these traces  $\tau^{(b,r)}$ $b \in B_{\gamma}$,
$r=0,1,2,\dots$ and $\tau^{\infty}$ 
 exhaust  the extreme traces on the core $\F^{(\infty)}$.

For the classification of extreme traces, the analysis of point
masses is  essential.  A tracial state on the core gives 
the  traces on compact algebras $K(X^{\otimes n})$ by restriction. 
We get a certain relation among the point masses 
of the measures on $K$ corresponding 
to the traces on $K(X^{\otimes n})$ and $K(X^{\otimes n+1})$ 
by Morita equivalence. 
The difficulty of the analysis comes from the fact that 
$K(X^{\otimes n})$ is {\it not} included in $K(X^{\otimes n+1})$. 
Thus the necessary condition can not be described by simple 
restriction. Here comes a technical use of a countable 
basis and a result of the extension  of traces on a subalgebra and
an ideal to their sum modified in \cite{KW2} following after Exel and Laca 
\cite{EL}.

For a tracial state
of the core, we remove the discrete parts of it and get a finite trace
which has no point mass.  By the argument of principle of proper contraction 
on tracial state space, we can show
that  any tracial state on the core without  point mass on $A = C(K)$ 
is exactly the trace associated with the 
Hutchinson measure.  Combining these, we can present a complete list of 
of extreme traces on the core. This gives a complete classification of 
the traces.  

In \cite{IKW} we completely classified the KMS states for the gauge action 
on the  associated \cst-algebra $\O_{\gamma}$. 
We shall show that which trace on the core is extended to a KMS state on  
 $\O_{\gamma}$. We recommend  a paper \cite{KR} by Kumjian and J. Renalut 
for a general study of the KMS states on C${}^*$-algebras 
associated to expansive maps. In many cases, the inverse braches of 
an expansive map gives a self-similar map in our study.

The content of the paper is as follows:  In section 2,
we recall some basic notations and typical examples of 
self-similar maps. In section 3, we
give some  preliminary results for   \cst
-correspondences.  In section 4, we describe Rieffel correspondence of 
{\it finite} traces between \cst -algebras which are 
Morita equivalent by an equivalence bimodule of finite degree type 
using countable bases.  
In section 5, we apply a result of trace extension
to our situation.  We get a certain  relation between point masses 
of the measures
on $K$ corresponding to the traces on compact algebras of the tensor
products of the \cst -correspondences.  In section 6, we construct model 
traces on the core .  We also need a trace extension result for
our construction.  In section 7, we classify the extreme  traces on the
core.  Since the set of branched points is assumed to be finite, we can
reduce the classificaion of the traces 
to the uniqueness of the invariant measure without point mass. 
In fact, the final step of the classification is completed by 
the contraction principle on  the 
metric space of the probability measures using  the Lipschitz norm 
and the fact that the diameter of it is finite.
 
\section
{Self-similar maps}
Let $(\Omega,d)$ be a (separable) complete  metric space.  A map 
$f:\Omega \rightarrow \Omega$ is called
a proper contraction if there exists a constant $c$ and $c'$ with
 $0<c' \leq c <1$ such that
$0<  c'd(x,y) \leq d(f(x),f(y)) \le cd(x,y)$ for any  $x$, $y \in \Omega$.

We consider 
a family $\gamma = (\gamma_1,\dots,\gamma_N)$ of $N$ proper contractions 
on $\Omega$. We assume that $N \geq 2$. Then there exists a unique non-empty compact set  
$K \subset \Omega$ which is self-similar in the sense that 
$K = \bigcup_{i=1}^N \gamma_i(K)$. See 
Falconer \cite{F} and Kigami \cite {Kig} for more on 
fractal sets. 

In this note we usually forget an ambient space $\Omega$ as in \cite{KW1} 
and start 
with the following setting: Let $(K,d)$ be a compact metric set and  
 $\gamma = (\gamma_1,\dots,\gamma_N)$ is a family of $N$ proper contractions 
on $K$. We say that 
$\gamma$ is a self-similar map on $K$ if
$K = \bigcup_{i=1}^N \gamma_i(K)$.

\begin{defn} 
We say that $\gamma$ satisfies the open set condition if there exists an
open subset $V$ of $K$ such that $\gamma_j(V) \cap \gamma_k(V)=\phi$ for
$j \ne k$ and  $\bigcup_{i=1}^N \gamma_i(V) \subset V$. Then  
$V$ is an open dense subset of $K$. See a book \cite{F} by Falconer, for 
example. 
\end{defn}

Let $\Sigma  = \{\,1,\dots,N\}$.  For $k \ge 1$, we put $\Sigma^k = \{\,1,\dots,N\}^k$.
For a self-similar map $\gamma$ on a compact metric space $K$, we introduce 
the following subsets of $K$: 
\begin{align*}
B_{\gamma} = &\{\,b \in K\,|\, b =\gamma_{i}(a)=\gamma_{j}(a), \,\,
\text{for some}\,\, a \in K\, \text{and } \, i \ne j,\,\, \}, \\
C_{\gamma} = & \{\,a\in K\,|\,\gamma_j(a) \in B_{\gamma}\,\,\text{for
some}\,\,j\,\},\\
P_{\gamma} = & \{\,a\in K\,|\, \exists k \ge 1,\,\exists (j_1,\dots,j_k)
\in \Sigma^k\,\,
 \text{such that}\,\, \gamma_{j_1}\circ \cdots \circ \gamma_{j_k}(a) \in
B_{\gamma}\}, \\
O_{b,k}  = &\{\,\gamma_{j_1}\circ \cdots \circ \gamma_{j_k}(b)\,|\,
           (j_1,\dots,j_k) \in \Sigma^{k}\,\}\quad (k \ge 0),\quad
O_b    =  \bigcup_{k=0}^{\infty}O_{b,k}, \text{where} \,  O_{b,0} =\{b\}, \\
Orb = & \bigcup_{b \in B_{\gamma}} O_b. 
\end{align*}

We call $B_{\gamma}$ the branch set of $\gamma$, $C_{\gamma}$ the branch
value set of $\gamma$ and $P_{\gamma}$ the  postcritical set of $\gamma$.
We define branch index at $(\gamma_j(y),y)$ by 
$e_{\gamma}(\gamma_j(y),y) = \verb!#!\{i \in
\Sigma |
\gamma_j(y)=\gamma_i(y)\}$.
We call $O_{b,k}$ the $k$-th $\gamma$ orbit of $b$, and $O_b$ the $\gamma$
orbit
of $b$.

{\bf Assumption A.} 
 Throughout the paper, we assume that a self-similar map 
$(K,\gamma)$ satisfies the following for simplicity:
\begin{enumerate}
 \item Their exists a continuous map $h:K \rightarrow K$ such that 
$h(\gamma_j(y))=y$ for $y \in K$ and $j=1,\dots,N$, and 
$h^{-1}(y) = \bigcup_{j = 1}^{N} \gamma_j(y)$. 
 \item The sets $B_{\gamma}$ and $P_{\gamma}$ are finite sets.
 \item  The intersection of $B_{\gamma}$ 
and $\bigcup_{b \in B_{\gamma}} \bigcup_{k = 1}^{\infty} O_{b,k} $ is empty.
\end{enumerate}

Under the assumption A above, $B_{\gamma}$ is the 
branch set of $h$ and $C_{\gamma} = h(B_{\gamma})$ is the branch value 
of $h$. Moreover 
$P_{\gamma} = \bigcup_{k=1}^{\infty} h^k(B_{\gamma})$, 
$O_{b,k} = h^{-k}(b)$ and  
$O_b    =  \bigcup_{k=0}^{\infty} h^{-k}(b)$ 
the backward orbit of $b$ under $h$. The condition (3) of the above 
assumption A means that for any branch point $b$, 
the backward orbit of $b$ under $h$ does not intersect with the 
set of branch points except the 0-th orbit $b$. The condition (3) is 
not trivial but satisfied in many cases including a tent map. The 
condition (3) can be replaced by the following condition (3)': \\

(3)' $B_{\gamma} \cap   P_{\gamma}$ is empty.  \\

In fact, not(3) is equivalent to that there exists 
$y \in B_{\gamma} \cap ( \bigcup_{k=1}^{\infty} h^{-k}(b))$. 
This means that  there exist $y,b \in B_{\gamma}$ and an integer $k \geq 1$ 
such that $h^k(y) = b$. This is equivalent to not(3)'. 

Moreover,  since we assume that  $P_{\gamma}$ is finite,  $(K,\gamma)$ satisfiesthe open set condition.  Let  $V=K \backslash
P_{\gamma}$. Then 
$V$ is an open dense subset of $K$ satisfying $\gamma_j(V) \cap
\gamma_k(V) = \emptyset$ for $j \ne k$. In fact, on the contrary suppose that 
$\gamma_j(V) \cap
\gamma_k(V) \not= \emptyset$. Then there exist $u,v \in V$ such that 
$b := \gamma_j(u) =  \gamma_k(v)$. Then $h(b) = u = v$ and 
$u \in C_{\gamma} \subset P_{\gamma}$, which is a contradiction. 
We show that $V$ is invariant under $\gamma$. 
On the contrary suppose that
 $x =\gamma_j(u)$ is not in $V$ for 
some $u \in V$. Then $x$ is  in $P_{\gamma}$ and $u = h(x)$ is not in 
$P_{\gamma}$. Hence $x = h^k(b)$ for some $b \in B_{\gamma}$ and a 
natural number $k$, so  
$u = h^{k+1}(b)$, which contradicts that $u$ is not in $P_{\gamma}$.

\begin{exam} \label{ex:tent} {\rm (tent map) }
Let $K=[0,1]$, $\gamma_1(y) = (1/2)y$ and $\gamma_2(y)
= 1-(1/2)y$.  
 Then a family $\gamma = (\gamma_1,\gamma_2)$ of proper contractions 
is a self-similar map.
We note that $B_{\gamma}=\{\,1/2\,\}$, $C_{\gamma} = \{\,1\,\}$ and 
$P_{\gamma} = \{\,0,1\,\}$.
A continuous map $h$ defined by
\[
 h(x) = \begin{cases}
         2x & \quad  0 \le x \le 1/2  \\
         -2x + 2 & \quad 1/2 \le x \le 1
        \end{cases}
\]
satisfies Assmption A (1). The map $h$ is the ordinary tent map
on $[0,1]$, and
$(\gamma_1,\gamma_2)$ is the inverse branches of the tent map $h$.
We note that $B_{\gamma}=\{\,1/2\,\}$, $C_{\gamma} = \{\,1\,\}$ and 
$P_{\gamma} = \{\,0,1\,\}$. We see that $h(1/2) = 1, h(1) = 0, h(0) = 0$. 
Hence a self-similar map $\gamma = (\gamma_1,\gamma_2)$ satisfies 
the assumption A above. 
\end{exam}

\begin{exam}
Let $K=[0,1]$.  For $n \geq 2$, we take $n+1$ numbers $t_i$ $(i=0,\dots,n)$ 
such that
$0=t_0 < t_1 < \dots < t_{n-1}<t_n = 1$.  Let $h$ be a continuous map
from $K$ to $K$ such that $h$ is linear on each subinterval
$[t_i,t_{i+1}]$ and takes value $0$ or $1$ at the endpoints of the
subinterval interchangingly.  We denote by $\gamma_i$ $(i=1,\dots,n)$ 
the inverse
branches of $h$.  Then $\gamma = (\gamma_1,\dots,\gamma_N)$ satisfies
the assumption A.  
We note that $B_{\gamma} = \{\,t_1,\dots,t_{n-1}\}$,
and $P_{\gamma}=\{\,0,1\,\}$. If $n = 2$, then $C_{\gamma} = \{0\}$ or 
$ \{1\}$. If $n \geq 3$, then $C_{\gamma} = \{\,0,1\,\}$.
\end{exam}

\begin{exam} \cite{KW1} {\rm (Koch curve)}
Let $\omega = \frac{1}{2}+i\frac{\sqrt{\,3}}{6} \in \comp$.  Consider
the two contraction $\tilde{\gamma}_1$, $\tilde{\gamma}_2$ on the
triangle domain
$\Delta \subset \comp$ with vertices $\{\,0,\omega,1\,\}$
defined by $\tilde{\gamma}_1(z) = \omega \overline{z}$ and
$\tilde{\gamma}_2(z)
=(1-\omega)(\overline{z}-1)+1$ for $z \in \comp$.
Let $\tau$ be the reflection in the line $x=1/2$.  We put
$\gamma_1 = \tilde{\gamma}_1$ and $\gamma_2 = \tilde{\gamma}_2\circ
 \tau$.  Put $K = \bigcap_{n=1}^{\infty}\bigcap_{(j_1,\dots,j_n)\in
 \Sigma^n} \gamma_{j_1}\circ \cdots \circ \gamma_{j_n}(\Delta)$.  Then
 $K$ is the Koch Curve and $(K,\gamma)$ is a self-similar map satisfying
 the assumption A. We note that $B_{\gamma} = \{\omega \}$, 
$C_{\gamma} = \{1\}$ and 
$P_{\gamma} = \{0,1\}$. 
\end{exam}

\begin{exam}\cite{KW1} {\rm (Sierpinski gasket)}
Let $P=(1/2, \sqrt{\,3}/2)$, $Q=(0,0)$, $R=(1,0)$,
$S=(1/4,\sqrt{\,3}/4)$, $T=(1/2,0)$ and $U=(3/4,\sqrt{\,3}/4)$.
Let $\tilde{\gamma}_1$, $\tilde{\gamma}_2$ and $\tilde{\gamma}_3$ be
contractions on the regular triangle $T$ on ${\bf R}^2$ with three
vertices $P$, $Q$ and $R$ such that
\[
\tilde{\gamma}_1(x,y) =
\left(\frac{x}{2}+\frac{1}{4},\frac{1}{2}y\right), \quad
\tilde{\gamma}_2(x,y)=\left(\frac{x}{2},\frac{y}{2}\right), \quad
\tilde{\gamma}_3(x,y) = \left(\frac{x}{2}+\frac{1}{2},\frac{y}{2}\right).
\]
We denote by $\alpha_{\theta}$ a rotation by angle $\theta$.
We put $\gamma_1 = \tilde{\gamma}_1$, $\gamma_2 = \alpha_{-2\pi/3}\circ
\tilde{\gamma}_2$, $\gamma_3 = \alpha_{2\pi/3}\circ \tilde{\gamma}_3$.
Then  $\gamma_1(\Delta PQR) = \Delta PSU$, 
$\gamma_2(\Delta PQR) = \Delta TSQ$ and 
$\gamma_3(\Delta PQR) = \Delta TRU$, 
where $\Delta ABC$ denotes the regular triangle whose vertices
are A, B and C.
Put $K = \bigcap_{n=1}^{\infty}\bigcap_{(j_1,\dots,j_n)\in \Sigma^n}
(\gamma_{j_1}\circ \cdots \circ \gamma_{j_n})(T)$.
Then $(K,\gamma)$ is a self similar map satisfying assumption A, and $K$
is the Sierpinski gasket.  
$B_{\gamma}=\{\,S,T,U\,\}$, 
$C_{\gamma} = P_{\gamma}=\{\,P,Q,R\,\}$ and $h$ is given
 by
\[
 h(x,y) =
\begin{cases}
    &  \gamma_{1}^{-1}(x,y)  \quad (x,y) \in \Delta PSU \cap K\\
    &  \gamma_{2}^{-1}(x,y)  \quad (x,y) \in \Delta TSQ \cap K \\
    &  \gamma_{3}^{-1}(x,y)  \quad (x,y) \in \Delta TRU \cap K,
\end{cases}
\]

\end{exam}

\section {Hilbert $C^*$-bimodules}

Following \cite{KW1}, we introduce Hilbert $C^*$-bimodules 
(or \cst-correspondences) and their Cuntz-Pimsner
\cst -algebras
for self-similar maps.

Let $A$ be a $C^*$-algebra and $X$ a Hilbert right $A$-module.  
Let $\L(X)$ the set of bounded linear operators on $X$ which have
adjoints with respect to the $A$-inner product $(x|y)_A$. 
We denote by $\theta_{x,y}$ the "rank one"
operator given by $\theta_{x,y}(z) = x(y|z)_A$ for 
$x,y, z \in X$. We denote  by $\K(X)$ the
\cst -algebra generated by  rank one operators, which is sometimes called 
the compact ideal.

We say that 
$X$ is a Hilbert bimodule over a $C^*$-algebras $A$ 
if $X$ is a Hilbert right  $A$-
module with a *-homomorphism $\phi : A \rightarrow L(X)$.  
We always assume 
that $X$ is full and $\phi$ is injective.

We recall the construction of the Cuntz-Pimsner \cst -algebra 
$\O_X$ associated with $X$.  
Let $F(X) = \bigoplus _{n=0}^{\infty} X^{\otimes n}$
be the full Fock module of $X$ with a convention $X^{\otimes 0} = A$. 
 For $\xi \in X$, the creation operator $T_{\xi} \in L(F(X))$ is defined by 
\[
T_{\xi}(a) =  \xi a  \qquad \text{and } \ 
T_{\xi}(\xi _1 \otimes \dots \otimes \xi _n) = \xi \otimes 
\xi _1 \otimes \dots \otimes \xi _n .
\]
We define $i_{F(X)}: A \rightarrow L(F(X))$ by 
$$
i_{F(X)}(a)(b) = ab \qquad \text{and } \ 
i_{F(X)}(a)(\xi _1 \otimes \dots \otimes \xi _n) = \phi (a)
\xi _1 \otimes \dots \otimes \xi _n 
$$
for $a,b \in A$.  The Cuntz-Toeplitz algebra ${\mathcal T}_X$ 
is the C${}^*$-algebra acting on $F(X)$ generated by $i_{F(X)}(a)$
with $a \in A$ and $T_{\xi}$ with $\xi \in X$. 

Let $j_K | K(X) \rightarrow {\mathcal T}_X$ be the homomorphism 
defined by $j_K(\theta _{\xi,\eta}) = T_{\xi}T_{\eta}^*$. 
We consider the ideal $I_X := \phi ^{-1}(K(X))$ of $A$. 
Let ${\mathcal J}_X$ be the ideal of ${\mathcal T}_X$ generated 
by $\{ i_{F(X)}(a) - (j_K \circ \phi)(a) ; a \in I_X\}$.  Then 
the Cuntz-Pimsner algebra ${\mathcal O}_X$ is defined as 
the quotient ${\mathcal T}_X/{\mathcal J}_X$ . 
Let $\pi : {\mathcal T}_X \rightarrow {\mathcal O}_X$ be the 
quotient map.  
We set $S_{\xi} = \pi (T_{\xi})$ and $i(a) = \pi (i_{F(X)}(a))$. 
Let $i_K : K(X) \rightarrow {\mathcal O}_X$ be the homomorphism 
defined by $i_K(\theta _{\xi,\eta}) = S_{\xi}S_{\eta}^*$. Then 
$\pi((j_K \circ \phi)(a)) = (i_K \circ \phi)(a)$ for $a \in I_X$.   

The Cuntz-Pimsner algebra ${\mathcal O}_X$ is 
the universal C${}^*$-algebra generated by $i(a)$ with $a \in A$ and 
$S_{\xi}$ with $\xi \in X$  satisfying that 
$i(a)S_{\xi} = S_{\phi (a)\xi}$, $S_{\xi}i(a) = S_{\xi a}$, 
$S_{\xi}^*S_{\eta} = i((\xi | \eta)_A)$ for $a \in A$, 
$\xi, \eta \in X$ and $i(a) = (i_K \circ \phi)(a)$ for $a \in I_X$.
We usually identify $i(a)$ with $a$ in $A$.  
We also identify $S_{\xi}$ with $\xi \in X$ and simply write $\xi$
instead of $S_{\xi}$.  
There exists an action 
$\beta : {\mathbb R} \rightarrow \Aut \ {\mathcal O}_X$
defined by $\beta_t(\xi) = e^{it}\xi$ for $\xi\in X$ and $\beta_t(a)=a$ 
for $a\in A$, which is called the {\it gauge action}.

Let  $\gamma$  be a self-similar map  on a compact metric space $K$ 
satisfying Assumption A. 
Let $A = {\rm C}(K)$, $\C=\bigcup_{i=1}^N\{\,(\gamma_i(y),y)\,|\,y \in
K\,\}$,
and $X = {\rm C}(\C)$.
For $f$, $g \in X$ and $a$, $b \in A$, we define left and right $A$
module actions on
$X$ and an $A$-inner product by
\begin{align*}
 (a\cdot f \cdot b)(\gamma_j(y),b)  = &
a(\gamma_{j}(y))f(\gamma_j(y),y)b(y) \quad
  y \in K, \quad j=1,\dots,N  \\
 (f|g)_A(y) = & \sum_{j=1}^N\overline{f(\gamma_j(y),y)}g(\gamma_j(y),y)
 \quad y \in K.
\end{align*}

We define the representation $\phi$ of $A$ in $\L(X)$ by
$\phi(a)f=a\cdot f$.
By \cite{KW1}, $(X,\phi)$ is proved to be a $C^*$-correspondence over
$A$.  We denote by
$\O_{\gamma}$ the Cuntz-Pimsner $C^*$-algebra associated with $X$ and 
call it the Cuntz-Pimsner algebra $\O_{\gamma}$ 
associated with a self-similar map ${\gamma}$.

We put $J_X=\phi^{-1}(\K(X))$. Then $J_X$ is an ideal of $A$. 

We recall some  basic facts on the Cuntz-Pimsner algebra $\O_{\gamma}$ 
associated with a self-similar map ${\gamma}$.  

\begin{lemma} \cite{KW1}
Let  $\gamma$  be a self-similar map  on a compact metric space $K$ 
If $(K,\gamma)$ satisfies the  open set condition, then 
the associated Cuntz-Pimsner algebra
$\O_{\gamma}$ is simple and purely infinite.  Moreover $J_X$ 
remembers the branch set $B_{\gamma}$ so that 
$J_X
= \{\,f \in A\,|\,f(b)=0 \quad \text{for each }b \in B_{\gamma}\,\}$

\end{lemma}

\par
Let $X^{\otimes n}$ be the $n$-times inner tensor product of $X$ and
$\phi_n$ denotes the left module action of $A$ on $X^{\otimes n}$.
Put 
$$
\F^{(n)} = A \otimes I + \K(X)\otimes I + \K(X^{\otimes 2})\otimes I 
 \cdots + \K(X^{\otimes n}) \subset \L(X^{\otimes n})
$$
We embedd $\F^{(n)}$ into $\F^{(n+1)}$ by $T \mapsto T\otimes I$ for 
$T \in \F^{(n)}$. Put
$\F^{(\infty)} = \overline{\bigcup_{n=0}^{\infty}\F^{(n)}}$.

By Pimsner \cite{Pi} we can identify $\F^{(n)}$ with the  
$C^*$-subalgebra of  $\O_{\gamma}$ generated by $A$ and $S_xS_y^*$ 
for $x,y \in X^{\otimes k}$, $k=1,\dots, n$ 
under identifying $S_xS_y^*$  with 
$\theta_{x,y}$ .  
Then the inductive limit algebra $\F^{(\infty)}$ is exactly the fixed point subalgebra of $\O_{\gamma}$
under the gauge
action.  We call $\F^{(\infty)}$ the core of $\O_X$.
We note that $\F^{(n+1)} = \F^{(n)} \otimes I + \K(X^{\otimes n+1})$. 
Thus  $\F^{(n)}$ is
a $C^*$-subalgebra  of $\F^{(n+1)}$ containing unit and $\K(X^{\otimes
n+1})$ is an ideal of $\F^{(n+1)}$. We sometimes write 
$\F^{(n+1)} = \F^{(n)} + \K(X^{\otimes n+1})$ for short.

\section
{The Rieffel correspondence of traces}
In this section, we study the Rieffel correspondence of traces 
on Morita equivalent $C^*$-algebras $A$ and $B$. If $C^*$-algebras 
are not unital, then there does not exist a bijective correspondence 
between the {\it finite} traces on Morita equivalent 
$C^*$-algebras $A$ and $B$ in general. For example, consider 
$A = {\mathbb C}\oplus {\mathbb C}$ and 
$B  = {\mathbb C}\oplus K(H)$. We consider the case that
a Hilbert module is of finite degree type and show the following: 
If $C^*$-algebras 
$A$ and $B$ are Morita equivalent by 
an equivalence bimodule 
of finite degree type for both sides, then 
there exists a bijective correspondence 
between the {\it finite} traces on $A$ and $B$.
In general,  
if $C^*$-algebras 
$A$ and $B$ are Morita equivalent,  then 
there exists a bijective correspondence 
between the densely defined lower semi-continuous 
traces on $A$ and $B$ written by countable bases, 
see, for example,   N. Nawata \cite{Na}.

Since the correspondence can be explicitely written  using  basis, 
we recall the notion of basis for  a Hilbert $C^*$-module,  
see \cite{KPW1} and \cite{K}. 
Let $X = X_A$ be a (right) Hilbert $C^*$-module $X$ over a $C^*$-algebra $A$. 
A family $({u_i})_{i \in I}$ in $X$ is called a (right) basis,  
(or a tight frame more precisely as in \cite{FL}) of $X$ 
if 
$$
x = \sum _{i \in I} u_i(u_i|x)_A  \text{ for any } x \in X, 
$$ 
where the sum is taken as unconditional norm convergence. 
Furthermore  $({u_i})_{i \in I}$ is called a finite basis if 
$(u_i)_{i \in I}$ is  a finite set. 
If a Hilbert $C^*$-module is countably generated, 
then there  exists a countable basis (that is, 
finite or a countably infinite basis ) of $X$ and written  
as $\{u_i\}_{i=1}^{\infty}$, where some $u_i$ may be zero. 
Let $\{u_i\}_{i=1}^{\infty}$ be a countable basis of 
 Hilbert right $B$-module $X$ and
$\{v_j\}_{j=1}^{\infty}$ a countable basis of Hilbert right 
$C$-module $Y$. Assume that there exists a $*$-homomorphisms
$\phi : B \rightarrow L(Y)$. 
Then $\{u_i\otimes v_j\}_{i=1,2,\dots,j=1,2,\dots}$ constitutes a basis
of $X \otimes_B Y$ for any
ordering (\cite{KPW1} and \cite{KW2}).

Let $A$ be a $C^*$-algebra.  
We denote by $\T(A)$ the set of tracial states
on $A$.  We assume that $\T(A)$ is not empty.
Let $X$ be a countably generated
Hilbert $A$-module and $\{u_i\}_{i=1}^{\infty}$ a basis of $X$.
For a tracial state $\tau$ on $A$, $\sup_n \sum_{i=1}^n
\tau((u_i|u_i)_A) \in [0,\infty]$ does not depend on the choice of basis
$\{u_i\}_{i=1}^{\infty}$ as in \cite{KW2}, see also 
the proof of Proposition \ref{prop:traces} of this paper. 
We put $d_{\tau} = \sup_n \sum_{i=1}^n
\tau((u_i|u_i)_A)$.  
We call that $X$ is of {\it finite degree type} if 
$\sup_{\tau \in \T(A)}d_{\tau}< \infty$ 
(\cite {KPW1} and \cite{KW2}). 
For example, 
let $X$ be the Hilbert bimodule  associated with a self-similar map 
${\gamma}= (\gamma_1,\dots,\gamma_N)$ of $N$ proper contractions. 
Then $X_A$ is of finite degree type and 
$\sup_{\tau \in \T(A)}d_{\tau} = N$. 

In the papr we assume that a Hibert $C^*$-module is countably
generated. Let $A$ and $B$ be $C^*$-algebras. 
We say that 
an $B-A$ equivalence bimodule $X = {}_BX_A$ is 
of finite degree type for both sides if 
 $\T(A)$ and $\T(B)$ are not empty  and
$X_A$ and ${}_BX$ are of finite degree type, where 
$X_A$ and ${}_BX$ are Hilbert $C^*$-modules 
considered as a right
$A$-module and a left $B$-module respectively.

\begin{prop} \label{prop:traces}
If $C^*$-algebras 
$A$ and $B$ are Morita equivalent by 
an equivalence bimodule $X$ 
of finite degree type for both sides, then 
there exists a bijective correspondence 
between the {\it finite} traces on $A$ and $B$. 
Let $\{u_i\}_{i=1}^{\infty}$ and $\{w_j\}_{j=1}^{\infty}$ be
countable bases of $X_A$ and ${}_BX$ respectively. 
The correspondence is given explicitely as follows: \\
For a finite trace $\tau$ on $A$, we can define a finite trace
$\tilde{\pi}(\tau)$ on $B$  by
\[
 \tilde{\pi}(\tau)(b) = \sum_{i=1}^{\infty}\tau((u_i|bu_i)_A).
\]
For a finite trace $\sigma$ on $B$, we can define a
finite trace $\tilde{\psi}(\sigma)$ on $A$ by
\[
 \tilde{\psi}(\sigma)(a) = \sum_{j=1}^{\infty} \sigma({}_B(w_j a|w_j)).
\]
The correspondence  does not depend on the choice of bases. 
The correspondence preserves the order of traces.
Moreover we have
\[
  \tilde{\pi}(\tau)({}_B(x|y)) = \tau((y|x)_A), \ \ 
  \tilde{\psi}(\sigma)((z|w)_A) = \sigma({}_B(w|z)). 
\]
\end{prop}
\begin{proof} Firstly, we consider the case that $B = K(X_A)$. 
Since $X$ is of finite degree type, for $T \in \K(X)_+$, we have
\[
 \sum_{i=1}^{\infty} \tau((u_i|Tu_i)_A) \le \|T\| d_{\tau}.
\]
For general $T \in \K(X)$, we have 
 $\sum_{i=1}^{\infty}
 |\tau((u_i|Tu_i)_A)| \le 4 d_{\tau}\|T\| < \infty$. Hence 
$\tilde{\pi}(\tau)$ is well-defined.  

Let $\{v_j\}_{j=1}^{\infty}$ be another basis of  $X_A$. 
Then for $T \in \K(X)$, 
we have 
\begin{align*}
 & \sum_{i=1}^{\infty}\tau((u_i|T^*Tu_i)_A) 
  = \sum_{i=1}^{\infty}\tau((Tu_i|Tu_i)_A) 
  = \sum_{i=1}^{\infty}\tau\left( (
\sum_{j=1}^{\infty}v_j(v_j|Tu_i)_A | Tu_i)_A \right) \\ 
 &=   \sum_{i=1}^{\infty} \tau\left(
 \sum_{j=1}^{\infty}(v_j|Tu_i)_A^* (v_j|Tu_i)_A \right) 
  =  \sum_{i=1}^{\infty}  \sum_{j=1}^{\infty}
\tau\left( (v_j|Tu_i)_A^*(v_j|Tu_i)_A \right) \\
& =   \sum_{j=1}^{\infty}  \sum_{i=1}^{\infty}
\tau\left( (v_j|Tu_i)_A^*(v_j|Tu_i)_A \right)  
   =  \sum_{j=1}^{\infty}  \sum_{i=1}^{\infty}
\tau\left( (v_j|Tu_i)_A (v_j|Tu_i)_A^*\right) \\
 & =   \sum_{j=1}^{\infty}\sum_{i=1}^{\infty}
 \tau((T^{*}v_j|u_i(u_i|T^{*}v_j)_A))  
= \sum_{j=1}^{\infty} \tau((T^{*}v_j|\sum_{i=1}^{\infty}
  u_i(u_i|T^{*}v_j)_A)_A)  \\
 &=  \sum_{j=1}^{\infty} \tau((T^{*}v_j|T^{*}v_j)_A) 
  = \sum_{j=1}^{\infty} \tau((v_j | TT^*v_j)_A).
\end{align*}
This shows that for $S \in {\K(X)}_+$
\[
 \sum_{i=1}^{\infty} \tau((u_i|Su_i)_A)
\]
does not depend on the choice of the basis $\{u_i\}_{i=1}^{\infty}$
by taking $T=S^{1/2}$.
Moreover, the same calculation shows that
$T \mapsto \sum_{i=1}^{\infty}\tau((u_i|Tu_i)_A)$
gives a finite trace on $\K(X_A)$ 
by the polarization identity.
Secondly we consider the general case. 
Since there exists an isomorphism of $B$ onto $\K(X_A)$, 
we can define a finite trace
$\tilde{\pi}(\tau)$ on $B$  by
\[
\tilde{\pi}(\tau)(b) = \sum_{i=1}^{\infty}\tau((u_i|bu_i)_A).
\]
Similary, for a finite trace $\sigma$ on $B$, we can define a
finite trace $\tilde{\psi}(\sigma)$ on $A$ by
\[
 \tilde{\psi}(\sigma)(a) = \sum_{j=1}^{\infty} \sigma({}_B(w_j a|w_j)).
\]
Since ${}_B(x|y)z = x(y|z)_A = \theta_{x,y}(z)$, 
\begin{align*}
 \sum_{i=1}^{\infty}\tau((u_i|{}_B(x|y)u_i)_A)
& = \sum_{i=1}^{\infty}\tau(u_i|x(y|u_i)_A)_A)
= \sum_{i=1}^{\infty}\tau((u_i|x)_A(y|u_i)_A) \\
& =  \sum_{i=1}^{\infty}\tau((y|u_i)_A(u_i|x)_A)
= \tau(y| \sum_{i=1}^{\infty} u_i(u_i|x)_A) 
= \tau((y|x)_A), 
\end{align*}
we have 
\[
  \tilde{\pi}(\tau)({}_B(x|y)) = \tau((y|x)_A).
\]
Similarly we have 
\[
  \tilde{\psi}(\sigma)((z|w)_A) = \sigma({}_B(w|z)). 
\]
Since the Hilbert bi-module is full on both sides,  
the correspondence is bijective. By definition, 
the correspondence preserves the order of traces. 

\end{proof}

\section
{Analysis for point mass}

Let $\gamma$ be a self-similar map on $K$ which satisfies the
assumption A.
We denote by $B(K)$ the set of bounded Borel functions on $K$,
and by $M(K)$ the set of finite Borel measures on $K$.
For a Borel function $a$ on $K$, we define a Borel function $\tilde{a}$
by
\[
\tilde{a}(y) =
\sum_{j=1}^N\frac{1}{e_{\gamma}(\gamma_j(y),y)}a(\gamma_j(y)) 
= \sum_{x \in h^{-1}(y)} a(x) . 
\]
Even if $a$ is continuous, $\tilde{a}$ is not continuous in general, 
because we do not count the multiplicity in $\{x \in h^{-1}(y)\}$.  

We usually identify a finite Borel measure $\mu$ with the associated 
 positive linear functional on $C(K)$. So, we may write as 
$\mu (a) := \int a  d \mu$ for $a \in C(K)$ for short. For a 
Borel set $B \subset K$, we may write $\mu(B) = \mu (\chi_B)$ if 
no confusion arize.

For a finite Borel measure $\mu$ on $K$, we define a finite Borel measure
$F(\mu)$ on $K$ by 
\[
F(\mu)(a) = \mu(\tilde{a})\ \ \ ( a \in {\rm C}(K)). 
\]
In particular, we have $F(\delta_y) = \sum_{x \in h^{-1}(y)}\delta_x$, 
where $\delta_y$ is the Dirac measure on $y$ and 
the sum on $x$  should be  taken without multiplicity. 
In fact 
$$  
F(\delta_y)(a) = \delta_y (\tilde{a})= \tilde{a}(y)
= \sum_{j=1}^N\frac{1}{e_{\gamma}(\gamma_j(y),y)}a(\gamma_j(y))
= \sum_{x \in h^{-1}(y)}\delta_x (a). 
$$

\begin{lemma} \label{lem:lem6}
Let $X$ be a $C^*$-correspondence
associated with a self-similar map $\gamma$ on K. 
Let $\{u_{i}\}_{i=1}^{\infty}$ be a countable basis of $X$.
Then for $a \in A$ and $y \in K$, we have
\[
 \sum_{i=1}^{\infty}(u_i|\phi(a)u_i)_A(y) = \tilde{a}(y) .
\]
\end{lemma}

\begin{proof} For $f \in X = {\rm C}(\C)$, we have 
$|f(x,y)| \leq \|f\|_2$ 
and the identity $\sum_{i=1}^{\infty}u_i(u_i|f)_A = f$ 
converges in norm $\|\ \|_2$. Hence the left side  
converges also pointwisely.
For each fixed $y \in K $, $k= 1,\dots, N$, we consider the 
value of $\sum_{i=1}^{\infty}u_i(u_i|f)_A = f$  at $(x,y) = (\gamma_k(y),y)$: 
\begin{align*}
 \lim_{n \to \infty} \left( \sum_{i=1}^{n}u_i(\gamma_k(y),y)(u_i|f)_A(y)
\right)
= & \lim_{n \to \infty} \left( \sum_{i=1}^{n}u_i(\gamma_k(y),y)
 \sum_{j=1}^N \overline{u_i(\gamma_j(y),y)}f(\gamma_j(y),y) \right)
  \\
= & f(\gamma_k(y),y). 
\end{align*}
For $j$ with $\gamma_j(y) \not= \gamma_k(y)$, 
we take $f \in X$ such that $f(\gamma_k(y),y)=1$ and
 $f(\gamma_j(y),y)=0$.  Then we have
\begin{align*}
 \lim_{n \to \infty} \left( \sum_{i=1}^{n}u_i(\gamma_k(y),y)
 \sum_{j=1}^N \overline{u_i(\gamma_j(y),y)}f(\gamma_j(y),y) \right)
=  &\lim_{n \to \infty}\left( \sum_{i=1}^{n} \overline{u_i(\gamma_k(y),y)}
e_{\gamma}(\gamma_k(y),y) u_i(\gamma_k(y),y) \right) \\
=& \sum_{i=1}^{\infty}e_{\gamma}(\gamma_k(y),y)|u_i(\gamma_k(y),y)|^2=1.
\end{align*}
Recall that the branch index at $(\gamma_j(y),y)$ is given by 
$e_{\gamma}(\gamma_j(y),y) = \verb!#!\{i \in
\Sigma |
\gamma_j(y)=\gamma_i(y)\}$ and 
and that $e_{\gamma}(\gamma_j(y),y)$ express the multiplicity.
For $a \in A$, we have
\begin{align*}
 & \sum_{i=1}^{\infty}(u_i|\phi_n(a)u_i)_A(y)
 = \sum_{i=1}^{\infty} \sum_{j=1}^N \overline{u_i(\gamma_j(y),y)}
   a(\gamma_j(y))u_i(\gamma_j(y),y)  \\
 & =  \sum_{i=1}^{\infty} \sum_{j=1}^N
a(\gamma_j(y))|u_i(\gamma_j(y),y)|^2  \\
& =  \sum_{j=1}^N \left(
 \frac{1}{e_{\gamma}(\gamma_j(y),y)}
a(\gamma_j(y))\sum_{i=1}^{\infty}
   e_{\gamma}(\gamma_j(y),y)|u_i(\gamma_j(y),y)|^2 \right)  \\
& =   \sum_{j=1}^N   \frac{1}{e_{\gamma}(\gamma_j(y),y)}
    a(\gamma_j(y)) 
 =  \tilde{a}(y).
\end{align*}
\end{proof}

We shall rephrase the lemma above as follows: 
 
\begin{cor} Let $X$ be the $C^*$-correspondence
associated with a self-similar map $\gamma$ on K. 
Let $\{u_{i}\}_{i=1}^{\infty}$ be a countable basis of $X$.
Then we have
\[
   \sum_{i=1}^{\infty}\mu((u_i|\phi(a)u_i)_A) = F(\mu)(a).
\]
Moreover let $X^{\otimes n}$ be $n$-th tensor product of $X$ and 
$\{v_{j}\}_{j=1}^{\infty}$ be a countable basis of $X^{\otimes n}$. 
Then we have
\[
   \sum_{j=1}^{\infty}\mu((v_j|\phi(a)v_j)_A) = F^n(\mu)(a).
\]

\end{cor}
\begin{proof}It is enough to recall that 
$X^{\otimes n}_A$ is
the Hilbert $C^*$-module associated with a self similar map
$\gamma^n=(\gamma_{i_1}\circ \dots \circ
\gamma_{i_n})_{(i_1,\dots,i_n) \in \Sigma^n}$ on $K$ as in 
\cite{KW1}. 
\end{proof}

The identity above is proved for a certain concrete 
basis in Kajiwara \cite{K} by a direct computation. 

Let $X$ be the $C^*$-correspondence
associated with a self-similar map $\gamma$ and 
$n$ an integer with $n \ge 1$.
We define $y_0^n \in X^{\otimes n}$ by
\[
y_0^n := (1/\sqrt{\,N})^{n}1_{\C}\otimes 1_{\C} \otimes \cdots
\otimes 1_{\C} \in X^{\otimes n}.
\]

Since  $(y_0^n|y_0^n)_A = I$, for $z \in X^{\otimes n}$ 
$$
{}_{K(X_A)}(z|y_0^n)y_0^n = \theta_{z,y_0^n}(y_0^n) = z(y_0^n|y_0^n)_A = z.
$$
This shows that 
a singleton set $\{y_0\}$  
is a finite basis for  the left $K(X_A)$-module ${}_{K(X_A)} X^{\otimes n}$. 

\begin{lemma} \label{lem:full} Let $X$ be the $C^*$-correspondence
associated with a self-similar map $\gamma$ and $n$ an integer with $n \ge 1$.
Then Hilbert $C^*$-modules ${X^{\otimes n}}_A$ and 
${}_{\K(X^{\otimes n})}X^{\otimes n}$ are of finite degree type and full.
\end{lemma}
\begin{proof}
First, we prove that ${X^{\otimes n}}_A$ is of finite degree type for
$n=1$.  Put $a = 1_K$, a constant function on $K$, in  Lemma \ref{lem:lem6}. 
Then we have 
\[
 \sum_{i=1}^{\infty}(u_i|u_i)_A(y)
   = \tilde{1}_K(y) \le N\cdot 1_K,
\]
Hnece, 
for any tracial state $\tau$ on $A$, we have
\[
\sum_{i=1}^{\infty}\tau((u_i|u_i)_A) \le N. 
\]
Thus  $X_A$ is of finite degree type.
For general $n$, the conclusion follows from the fact 
that $X^{\otimes n}_A$ is
the Hilbert $C^*$-module associated with a self similar map
$\gamma^n=(\gamma_{i_1}\circ \dots \circ
\gamma_{i_n})_{(i_1,\dots,i_n) \in \Sigma^n}$ on $K$ as in 
\cite{KW1}. 
Since $\{y_0^n\}$ constitutes a finite basis of 
${}_{\K(X^{\otimes n})}X^{\otimes n}$, it is clear that 
${}_{\K(X^{\otimes n})}X^{\otimes n}$
is of finite degree type.

Since  $(y_0^n,y_0^n)_A = 1_K$, $X^{\otimes n}_A$ is full.  By definition,
${}_{\K(X^{\otimes n})}X^{\otimes n}$ is full as a left $\K(X^{\otimes
 n})$-module.  
\end{proof}

Let $\{\xi_i\}_{i=1}^{\infty}$ be a basis for $X^{\otimes n}$.
By Proposition \ref{prop:traces} and Lemma \ref{lem:full},
for a finite Borel measure $\mu$, 
there exists a finte trace $\tilde{\pi}_n(\mu)$ on $\K(X^{\otimes n})$ 
such that 
\[
 \tilde{\pi}_n(\mu) (T) = \sum_{i=1}^{\infty}\mu((\xi_i|T\xi_i)_A).
\]
and for $\xi$, $\eta \in X^{\otimes n}$, 
\[
 \tilde{\pi}_n(\mu)(\theta_{\xi,\eta})  = \mu((\eta|\xi)_A).
\]
Since $y_0^n \in X^{\otimes n}$ constitutes a basis for the left Hilbert
$C^*$-module ${}_{\K(X^{\otimes n })}X^{\otimes n}$, 
for a trace $\tau$ on $\K(X^{\otimes n})$, there exists a finite
Borel measure $\tilde{\psi}(\tau)_n$ such that
\[
 \tilde{\psi}_n(\tau)(a) = \tau(\theta_{y_0^n a,y_0^n}).
\]
for $a \in A = C(K)$. 

We shall study the traces on the core
$\F^{(\infty)} = \overline{\bigcup_{n=0}^{\infty}\F^{(n)}}$, where 
 $\F^{(n)} = A \otimes I + \K(X)\otimes I + \\K(X^{\otimes 2})\otimes I 
 \cdots + \K(X^{\otimes n}) \subset \L(X^{\otimes n})$. 
We usually identify $\F^{(n)}$ with the  
$C^*$-subalgebra of  $\O_{\gamma}$ generated by $a \in A$ and $S_xS_y^*$ 
for $x,y \in X^{\otimes k}$, $k=1,\dots, n$,  by identifying $S_xS_y^*$  with 
$\theta_{x,y}$ and $a$ with $\phi(a)$.  

We note that $\F^{(n+1)} = \F^{(n)} \otimes I + \K(X^{\otimes n+1})$, 
where $\F^{(n)}$ is
a $C^*$-subalgebra  of $\F^{(n+1)}$ containing unit and $\K(X^{\otimes
n+1})$ is an ideal of $\F^{(n+1)}$.
Any finite trace $\tau$ on $\F^{(\infty)}$ is determined by a sequence 
$(\tau_{[n]})_n$ of its restrictions $\tau_{[n]}$
to $\F^{(n)}$. Each trace $\tau_{[n+1]}$ is determmied by 
its restriction $\tau_{[n]}$ to  $\F^{(n)}$ and the restriction 
$\tau_{n+1}$
to $\K(X^{\otimes n+1})$ which is described by a finite Borel
measure $\mu_{n+1}$ on $K$ under the Riefell correspondence.

Conversely a compatible  sequence $(\tau_{[n]})_n$ of traces 
on increasing subalgebras 
$$
A=\F^{(0)} \subset \F^{(1)}\subset \F^{(2)} 
\subset \dots  \F^{(n )}\subset \dots 
$$
gives a trace on the inductive limit $\F^{(\infty)}$. 
First we find a sequence $(\mu_n)_n$ of finite Borel measures 
on $K$,  which  gives a sequence $(\tau_n)_n$ of traces on $\K(X^{\otimes n})$ 
by the Riefell correspondence. We shall extend 
a trace $\tau_{[n]}$ on $\F^{(n )}$ and a trace $\tau_{n+1}$
on $\K(X^{\otimes n+1})$ to a trace  $\tau_{[n+1]}$ on 
$\F^{(n+1 )}$.

Therefore 
we need the following extension lemma  
of traces to construct traces 
on the core and classify them. 

Let $A$ be a $C^*$-algebra and $I$ be an ideal of $A$.  For a positive 
linear functional
$\varphi$ on $I$, we denote by $\overline{\varphi}$ the canonical extension
of $\varphi$ to $A$. Recall that for an approximate unit 
$(u_{\lambda})_{\lambda}$ of $I$ and $x \in A$, we have  
$$
\overline{\varphi}(x) = \lim _{\lambda}{\varphi}(xu_{\lambda}) 
$$

We refer \cite{B} for the property of the canonical extension of
states.

The following key lemma is  proved  in Proposition 12.5 of 
 Exel and Laca \cite{EL} for
state case, and is modified  in Kajiwara and Watatani \cite{KW2} for trace
case.

\begin{lemma} \label{lem:extension}
Let $A$ be a unital $C^*$-algebra.  Let $B$ a $C^*$-subalgebra of $A$ 
 containing the
unit and
$I$ an ideal of $A$ such that $A=B+I$.  Let $\tau$ be a finite trace on
$B$,
and $\varphi$ a finite trace on  $I$. Suppose that the following
conditions are
satisfied :
\begin{enumerate}
 \item $\varphi (x) = \tau(x) $ for $ x \in B \cap I$.
 \item $\overline{\varphi}(x)  \le \tau(x) $ for positive $ x \in B$.
\end{enumerate}
Then there exists a unique finite trace on $A$ which extends $\tau$ and
$\varphi$.
Conversely, if there exists an  finite trace on $A$, then its 
restrictions
$\tau$  and $\varphi$ on $B$ and $I$ must satisfy  the above 
conditions (1) and (2). 
\end{lemma}

In order to describe  an inductive construction of traces on the core 
smoothly, 
we need  to modify the Rieffel correspondence  of finite traces 
$\tau_n$ on 
$\K(X^{\otimes n})$ and 
finite measures $\mu_n$ on $K$ by a normalization as follows: 

\begin{lemma} \label{lem:modified Riefell correspondence}
Let $\gamma$ be a self-similar map on $K$. 
Fix a natural number $n$. 
For a finite Borel measure $\mu$ on $K$, there exist a unique finite 
trace $\tau = {\pi}_n(\mu)$ on $\K(X^{\otimes n})$ such that 
\[
  \tau (\theta_{\xi,\eta}) = \frac{1}{N^n}\mu((\eta |\xi)_A).
\]
for $\xi,\eta \in X^{\otimes n}$. 
Conversely, for a finite trace $\tau$ on $\K(X^{\otimes n})$, there
 exists a unique finite Borel measure $\mu = \psi_n(\tau)$ on $K$
such that  $\pi_n(\mu_n) = \tau_n$ and 
\[
\mu(a) = N^n \tau(\theta_{y_0^n a,y_0^n}),
\]
for $a \in A = C(K)$.  
\end{lemma}

\begin{proof} We can put
 $\displaystyle{\pi_n(\mu) = \frac{1}{N^n}\tilde{\pi}_n(\mu)}$,
  $\displaystyle{\psi_n(\tau) = N^n\tilde{\psi}_n(\tau)}$.
\end{proof}

In the following we shall investigate compatibility conditions
of finite traces 
$\tau_{[n]}$ on $\F^{(n)}$ in terms of 
finite traces $\tau_n$ on $\K(X^{\otimes n})$ and corresponding
finite measures $\mu_n$ on $K$. 

\begin{prop} Let $\gamma$ be a self-similar map on $K$. 
 Let $\mu$ be a probablity Borel measure on $K$, which is  identified with 
a normalized trace on $C(K)$. For 
a positive integer $r$, let   $\varphi  = {\pi}_r(\mu)$ 
be the  corresponding finite trace on $\K(X^{\otimes r})$ under the 
modified Riefell 
correspondence in Lemma \ref{lem:modified Riefell correspondence}. 
Consider the  canonical extension $\overline{\varphi}$ of $\varphi$ to 
$\L(X^{\otimes r})$. For an integer $k$ with $0 \leq k \leq r$, 
put $\mu_k = \frac{1}{N^{r-k}} F^{r-k}(\mu)$. 
Then the restriction  of $\overline{\varphi}$ to $\K(X^{\otimes k}) \otimes I$ 
is exactly the  corresponding finite trace  
on $\K(X^{\otimes k})$ for $\mu_k$  under the modified Riefell 
correspondence, that is, 
for $x,y \in X^{\otimes k}$, we have 
\[
\overline{\varphi}(\theta_{x,y} \otimes I) = \frac{1}{N^k}\mu_k((y|x)_A).
\]

\end{prop}

\begin{proof}
Let $(s_i)_i$ be a basis of $X^{\otimes k}$ and 
$(t_j)_j$ be a basis of $X^{\otimes (r-k)}$. Then 
$(s_i \otimes t_j)_{i,j}$ is a basis of 
$X^{\otimes r} = X^{\otimes k} \otimes X^{\otimes (r-k)}$ 
(\cite{KPW1}). 
Using Lemma \ref{lem:lem6} for $X^{\otimes (r-k)}$ and $F^{r-k}$, 
for $x,y \in X^{\otimes k}$,
we have 
\begin{align*}
 \overline{\varphi}(\theta_{x,y} \otimes I) 
= & \sum_{i=1}^{\infty}\sum_{j=1}^{\infty}
\varphi( (\theta_{x,y} \otimes I)
\theta_{s_i\otimes t_j,s_i\otimes t_j})
=  \sum_{i=1}^{\infty}\sum_{j=1}^{\infty}
\varphi(\theta_{\theta_{x,y}(s_i)\otimes t_j,s_i\otimes t_j})\\
= &  \sum_{i=1}^{\infty}\sum_{j=1}^{\infty}
\frac{1}{N^r} \mu ((s_i\otimes t_j | \theta_{x,y}(s_i)\otimes t_j)_A) 
=  \sum_{i=1}^{\infty}\sum_{j=1}^{\infty}
\frac{1}{N^k} \frac{1}{N^{r-k}} \mu ((t_j | (s_i|\theta_{x,y}(s_i))_At_j)_A)\\
= & \sum_{i=1}^{\infty}
\frac{1}{N^k} \frac{1}{N^{r-k}}(F^{r-k}\mu)((s_i|\theta_{x,y}(s_i))_A) 
= \sum_{i=1}^{\infty}
\frac{1}{N^k} \frac{1}{N^{r-k}}(F^{r-k}\mu)((s_i| x(y|s_i)_A)_A))\\ 
= & \sum_{i=1}^{\infty}
\frac{1}{N^k} \frac{1}{N^{r-k}}(F^{r-k}\mu)((s_i|x)_A(y|s_i)_A) 
= \sum_{i=1}^{\infty}
\frac{1}{N^k} \frac{1}{N^{r-k}}(F^{r-k}\mu)((y|s_i)_A(s_i|x)_A) \\
= & 
\frac{1}{N^k} \frac{1}{N^{r-k}}(F^{r-k}\mu)
((y|\sum_{i=1}^{\infty}s_i(s_i|x)_A))_A) 
= \frac{1}{N^k}\mu_k((y|x)_A). 
\end{align*}
\end{proof}

\begin{cor} \label{cor:compact}
Let $\tau$ be a finite trace on $\F^{(n+1)}$ and 
$\tau_{n}$
and $\tau_{n+1}$ be its restriction to 
$\K(X^{\otimes n})\otimes I$ and $\K(X^{\otimes {n+1}})$.
Let $\mu_{n} = \psi_n(\tau_{n})$ and 
$\mu_{n+1}= \psi_{n+1}(\tau_{n+1})$ be the corresponding 
 Borel measures on $K$ 
by modified Riefell correspondence. Then we have the 
following relations: 
\begin{enumerate}
\item 
For $T \in \K(X^{\otimes n})^{+} \subset \F^{(n)}$, we have  that
 \[
 \overline{\tau_{n+1}}(T\otimes I) =
 \pi_n\left(\left(\frac{1}{N}\right)F(\psi_{n+1}(\tau_{n+1}))\right)(T)
 \le \tau_{n}(T).
 \]
\item
For $f \in B(K)^+$, we have that 
 \[
 \psi_n(\overline{\pi_{n+1}(\mu_{n+1})}|\K(X^{\otimes n}))(f)
 = \frac{1}{N}F(\mu_{n+1})(f)
 \le \psi_n(\tau_{n})(f)=\mu_n(f)
 \]
\item
For $T\otimes I \in \K(X^{\otimes n}\otimes I) \cap \K(X^{\otimes n+1})$, we have that 
\[
\tau_{n}(T) = \pi_n\left(\left(\frac{1}{N}\right)F(\mu_{n+1}))\right)(T).
\]
\end{enumerate}
\end{cor}
\begin{proof}
Apply Lemma \ref{lem:extension} as $A = \F^{(n+1)}$, 
$B = \F^{(n)} \otimes I$ and $I = \K(X^{\otimes n+1})$. 
\end{proof}

The Corollary above is used to show a relation between point masses of 
$\mu_{n}$ and $\mu_{n+1}$ later.

The content of (2) in the Corolally above is expressed by the following
commutative diagram:
\[
 \begin{CD}
  \T(\K(X^{\otimes n+1})) @>{\beta _{n+1}}>>
\T(\K(X^{\otimes n})) \\
  @V{\psi_{n+1}}VV @V{\psi_{n}}VV \\
  M(K)  @>{\frac{1}{N}F}>>  M(K) 
 \end{CD}
\]
where $\beta _{n+1}$ is defined by 
$\beta _{n+1}(\tau_{n+1}) = \overline{\tau_{n+1}}|\K(X^{\otimes n})$

In particular, if $\mu_{n+1}$ is a Dirac measure $\delta_{y}$, 
we have the following Corollary, which was obtained in 
\cite{IKW}. 

\begin{cor}
\[
\psi_n((\overline{\pi_{n+1}(\delta_{y})}|\K(X^{\otimes n}))
 = \frac{1}{N}\sum_{x \in h^{-1}(y)} \delta_x
\]
where we do not count the multiplicity for the right hand side.
\end{cor}

For $f \in {\rm C}(K)$, we define $\alpha^n(f) \in {\rm C}(K)$ by
$\alpha^n(f)(x) = f(h^{n}(x))$.  Then it holds that
\[
  \xi f  = \phi_n(\alpha^n(f))\xi, \ (\xi \in X^{\otimes n}).
\]
For $T \in \K(X^{\otimes n})$, we need to find $f \in {\rm C}(K)$ such that
$(T \phi_n(\alpha_n(f))) \otimes I \in \K(X^{\otimes n})\otimes I 
 \cap \K(X^{\otimes n+1})$.
By Pimsner \cite{Pi}, it holds that 
\[
 \K(X^{\otimes n}) \otimes I \cap \K(X^{\otimes n+1}) = 
\K(X^{\otimes n}J_X) \otimes I.
\]
where $J_X :=\varphi ^{-1}(K(X)) =  \{\,f \in {\rm C}(K)\,|\,f(b)=0\ \text{ for } b \in B_{\gamma}\}$.

The right hand side is expressed as
\begin{align*}
 \K(X^{\otimes n}J_X)
= & {\rm C}^* \left\{\,\sum_{i}\theta_{\xi_ig_i,\eta_i f_i}\,|\,
 \xi_i,\, \eta_i \in X^{\otimes n},\, g_i,\, f_i \in J_X  \right\} \\
=&   {\rm C}^* \left\{\,\sum_{i}\theta_{\xi_i ,\eta_i f_i}\,|\,
 \xi_i,\, \eta_i \in X^{\otimes n},\, f_i \in J_X  \right\}.
\end{align*}
For $\xi ,\eta \in X^{\otimes n}$, we have that 
\[
 \theta_{\xi ,\eta f} = \theta_{\xi, \phi_n(\alpha^n(f))\eta }
 =  \theta_{\xi,\eta} \phi_n(\alpha^n(f)^*) \in \L(X^{\otimes n}).
\]
Since any $T \in \K(X^{\otimes n})$ is a liimit of  a sequence $(T_k)_k$ 
such that  $T_k = \sum_{p=1} \theta_{x_p,y_p}$ $(x_p, y_p \in X^{\otimes n})$ 
is a finite sum of "rank one " operators, 
we have obtain the following lemma: 
\begin{lemma}
Let $T \in \K(X^{\otimes n})$ and $f \in J_X$, then 
$(T\phi_n(\alpha_n(f)))\otimes I
 \in \K(X^{\otimes n})\otimes I \cap \K(X^{\otimes n+1})$.
\end{lemma}

The following Lemmas describe  relations  of point masses of the 
corresponding measures 
$\mu _n$ and $\mu _{n+1}$ on $K$ for a finite trace on  $\F^{(n+1)}$.

\begin{lemma}\label{lem:point}
Let $\tau$ be a finite trace on $\F^{(n+1)}$ and 
$\tau_{n}$
and $\tau_{n+1}$ be its restriction to 
$\K(X^{\otimes n}) \otimes I$ and $\K(X^{\otimes {n+1}})$.
Let $\mu_{n} = \psi_n(\tau_{n})$ and 
$\mu_{n+1}= \psi_{n+1}(\tau_{n+1})$ be the corresponding 
 Borel measures on $K$ 
by modified Riefell correspondence. Let $b \in K$ and $b= \gamma_i(a)$ 
for some $i$, that is, $h(b) =a$. 
Then we have the following: 
\begin{enumerate}
\item If $b$ is not in $B_{\gamma}$, then the point masses on $a$ and $b$ 
are related as  $\mu_{n+1}(\{ a\})= N\mu_n(\{ b\})$. 
\item If $b$ is in $B_{\gamma}$, then $\mu_{n+1}(\{a \}) \leq N\mu_n(\{b \})$. 
\end{enumerate}
 
\end{lemma}

\begin{proof}
(1)Suppose that $b$ is not in $B_{\gamma}$. 
There exist a sequence of compact neighborhoods $F_m$ of $b$ and
open  neighborhoods $V_m$ of $b$ such that $F_m \subset V_m$, $V_m$ does
not contain any element of $B_{\gamma}$ and $V_m$ converges to $b$ as
$m$ tends to $\infty$.
We take $f_m \in {\rm C}(K)^+$ such that $0 \le f_m \le 1$, $f_m=1$
on $F_m$ and $f_m=0$ on $\overline{V_m}^c$.
Since $f_m$ vanishes on $B_{\gamma}$,  we have $f_m$ is in $J_X$. 
Therefore 
\[
\theta_{y_0^n,y_0^n f_m} \otimes I \in \K(X^{\otimes n})\otimes I \cap \K(X^{\otimes n+1}). 
\]
Then by (3) in Corollary \ref{cor:compact}, it holds that
\[
\tau_n(\theta_{y_0^n,y_0^n f_m}) 
= \pi_n\left(\left(\frac{1}{N}\right)F(\mu_{n+1}))\right)
(\theta_{y_0^n,y_0^n f_m}).
\]
Then 
\[
\frac{1}{N^{n}}\mu_n((y_0^n f_m| y_0^n)_A) = 
 \frac{1}{N^{n+1}}F(\mu_{n+1})((y_0^n f_m| y_0^n)_A ).
\]
Since $(y_0^n | y_0^n)_A  =1$, 
$\frac{1}{N^{n}}\mu_n(f_m) = 
 \frac{1}{N^{n+1}}F(\mu_{n+1})(f_m)$. 
 Let $m \to \infty$,  then $f_m$ converges to the characterestic function 
$\chi_{\{b\}}$ on the one point set $\{b\}$. 
Since $\tilde{\chi_{\{b\}}}  = \chi_{\{a\}} $, 
\[
\mu_n(\{b\}) = \frac{1}{N}F(\mu_{n+1})(\{b\})
 = \frac{1}{N} \mu_{n+1}(\{a\}).
\]
(2)Just apply (2) in  Corollary \ref{cor:compact}.  
A similar  calculation directly 
implies that 
\[
\frac{1}{N}F(\mu_{n+1})(\chi_{\{b\}}) \le \mu_n(\chi_{\{b\}}). 
\]
Therefore we have 
$
\frac{1}{N} \mu_{n+1}(\{a\}) \leq \mu_n(\{b\})
$. 
\end{proof}

\section
{Construction of model traces on the core}
For each $b \in B_{\gamma}$ and $r=0,1,2,\dots$, we shall construct a model 
trace $\tau^{(b,r)}$    on $\F^{(\infty)}$.

We denote by $\delta_{b}$ the Dirac measure on ${b}$.  
For $0  \le i \le r$, we define Borel measures
$\mu^{(b,r)}_i$
on $K$ by
\[
  \mu_i^{(b,r)} = \frac{1}{N^{r-i}}F^{r-i}(\delta_{b}), 
\]

We note that
\[
 \mu^{(b,r)}_i (f) = \frac{1}{N^{r-i}}
\sum_{(j_1,j_2,\dots,j_r) \in \Sigma^{r-i}}
f(\gamma_{j_1}\circ \cdots \circ \gamma_{j_{r-i}}(b)), \ \ (f \in A = C(K))
\]
because $h^{-k}(b) $ does not contain the brancehd points for $k = 1,2,3, \dots$ 
by the Assumption A.  

For $i >r$ we define Borel measures
$\mu^{(b,r)}_i = 0 $
on $K$.

\begin{prop} Let $\gamma$ be a self-similar map on $K$ with assumption A. 
For each $b \in B_{\gamma}$ and $r=0,1,2,\dots$, there exists a unique 
finite trace $\tau^{(b,r)}$ on $\F^{(\infty)}$ such that 
$\tau^{(b,r)}|A = \mu^{(b,r)}_0$,  the restriction of $\tau^{(b,r)}$ to 
$K(X^{\otimes i})$ is $\pi_i(\mu^{(b,r)}_i)$ for  $0  \le i \le r$ and 
the restriction of $\tau^{(b,r)}$ to 
$K(X^{\otimes i})$ is zero  for  $i >  r$, that is, 
\[
\tau^{(b,r)}(f) = (\frac{1}{N^{r}}F^{r}(\delta_{b}))(f) = 
\frac{1}{N^{r}}
\sum_{x \in h^{-r}(b)} 
f(x)   \ \  \text{ for } \ \ f \in A
\]
and 
\[
\tau^{(b,r)}(\theta_{x,y}) = \frac{1}{N^{i}}\mu_i^{(b,r)} ((y|x)_A) 
\text{ for } \ \ x,y \in X^{\otimes i}, \ \ i<r . 
\]
\end{prop}

\begin{proof}
Recall that 
$\F^{(n)} = A \otimes I + \K(X)\otimes I + \K(X^{\otimes 2})\otimes I 
 \cdots + \K(X^{\otimes n}) \subset \L(X^{\otimes n})$. 
We identify $\F^{(n)}$ with the  
$C^*$-subalgebra of  $\O_{\gamma}$ generated by $A$ and $S_xS_y^*$ 
for $x,y \in X^{\otimes k}$, $k=1, \dots, n$
 by identifying $S_xS_y^*$  with 
$\theta_{x,y}$.  
We note that $\F^{(n+1)} = \F^{(n)} \otimes I + \K(X^{\otimes n+1})$. 
We shall construct 
finite traces $\tau^{(b,r)}_{[n]}$ on $\F^{(n)}$ $n=0,1,2,\dots$ 
by induction on $n$. We  choose a suitable sequence 
$(\mu_n^{(b,r)})_n$ of finite Borel measures 
on $K$,  which  gives a sequence $(\tau^{(b,r)}_n)_n$ 
of compatible traces on $\K(X^{\otimes n})$ 
by the modified Riefell correspondence. We shall extend 
a trace $\tau^{(b,r)}_{[n]}$ on $\F^{(n )}$ and a trace 
$\tau^{(b,r)}_{n+1}$
on $\K(X^{\otimes n+1})$ to a trace  $\tau^{(b,r)}_{[n+1]}$ on 
$\F^{(n+1 )}$ by Lemma \ref{lem:extension}.

We put $\tau^{(b,r)}_{[0]} = \tau^{(b,r)}_0 = \pi_0(\mu^{(b,r)}_0)$ and
$\tau^{(b,r)}_1=\pi_1(\mu^{(b,r)}_1)$.
Then $\tau^{(b,r)}_{[0]}$ is a tracial state. In fact, 
Let $(u_i)_i$ be a basis for $X^{\otimes r}$. Then 
\[
\tau^{(b,r)}_{[0]}(1) = \pi_0(\mu^{(b,r)}_0)
= (\frac{1}{N^{r}}F^{r}(\delta_{b}))(1) = 1
\]
We show that $\tau^{(b,r)}_{[0]}$ extends to a 
trace $\tau^{(b,r)}_{1}$ on $\F^{(1)}$.
We note that for $a \in A$
\[
 \overline{\tau_1^{(b,r)}}(a)
= \pi_0\left(\frac{1}{N}F(\psi_1(\tau_1^{(b,r)}))(a)\right)
= \tau_0^{(b,r)}(a),
\]
because we have
$\displaystyle{\frac{1}{N}F(\mu^{(b,r)}_1) = \mu^{(b,r)}_0}$. 
For $a \in A \cap \K(X)$, we have
\[
 \tau_1^{(b,r)}(a) = \overline{\tau_1^{(b,r)}}(a)
   =  \pi_0^{(b,r)}(a).
\]
By Lemma \ref{lem:extension},
we can construct a finite trace $\tau^{(b,r)}_{[1]}$ on $A+\K(X)$ which
extends $\tau^{(b,r)}_{[0]}$ and $\tau^{(b,r)}_1$.

Difine a trace  $\tau^{(b,r)}_i :=\pi_i(\mu^{(b,r)}_i)$ on 
$\K(X^{\otimes i})$ for  $0  \le i \le r$ and 
a trace  $\tau^{(b,r)}_i := 0$ on 
$\K(X^{\otimes i})$for  $i >  r$.

By induction on $n$, assume that we constructed a
finite trace $\tau^{(b,r)}_{[n]}$ on $\F^{(n)}$  
satisfying that 
$\left.\tau^{(b,r)}\right|_{\K(X^{\otimes i})}=\pi_i(\mu^{(b,r)}_i)$ 
for  $1  \le i \le n$ and $\left.\tau^{(b,r)}\right|_A= \mu^{(b,r)}_0$. 

We consider the three cases that $n<r$, $n = r$ and $n>r$.

(i)(Case that $n<r$):\\
By Corollarly \ref{cor:compact}, 
for $0 \leq i \leq n$ and $T \in \K(X^{\otimes i})$,  we have  that
\[
 \overline{\tau^{(b,r)}_{n+1}}(T\otimes I) 
=\pi_i\left(\left(\frac{1}{N^{n-i}}\right)F(\mu^{(b,r)}_{n+1}))\right)(T)
= \pi_i\left(\mu^{(b,r)}_{i})\right)(T) = 
\tau^{(b,r)}_{[n]}(T\otimes I).
\]
Thus 
$\overline{\tau^{(b,r)}_{n+1}}= \tau^{(b,r)}_{[n]} $ on 
$\F^{(n)}$.
Therefore 
by Lemma \ref{lem:extension},
we can extend to  a finite trace $\tau^{(b,r)}_{[n+1]}$ on
$\F^{(n+1)}$.

(ii)(Case that $n =r$):\\
Since we put $\tau^{(b,r)}_{r+1}=0$,  
clearly, $\overline{\tau^{(b,r)}_{r+1}}=0$ on
$\F^{(r+1)}$.
Hence 
$\overline{\tau^{(b,r)}_{r+1}}=0 \leq \tau^{(b,r)}_{[r]} $ on 
$\F^{(r)}$. 

  Let $T \in \K(X^{\otimes r})\otimes I  \cap \K(X^{\otimes r+1})
=\K(X^{\otimes r}J_X)$.  We need to show that $\tau^{(b,r)}_r(T)=0$.
Any $T \in \K(X^{\otimes r}J_X)$ is written as 
$T = \sum_{i=1}^{\infty}\theta_{\xi_i,\eta_i f_i}$ for 
$f_i \in J_X$.  Since $f_i(b)=0$ for $b \in B_{\gamma}$,
it holds that
\[
 \tau_r^{(b,r)}(T) = \sum_{i=1}^{\infty} \mu_r^{(b,r)}((\eta_i|\xi_i)_Af_i)
 =  \sum_{i=1}^{\infty} \mu_r^{(b,r)}((\eta_i|\xi_i)_Af_i)(b)  = 0.
\]
By Lemma \ref{lem:extension},
we can extend to  a finite trace $\tau^{(b,r)}_{[n+1]}$ on
$\F^{(n+1)}$.

(iii)(Case that $n >r$):\\
Since we put $\tau^{(b,r)}_{n+1}=0$ and 
$\tau^{(b,r)}_{n}=0$, we have that 
$\overline{\tau^{(b,r)}_{n+1}}=0$ on
$\F^{(n+1)}$. Hence 
$\overline{\tau^{(b,r)}_{n+1}}=0 \leq \tau^{(b,r)}_{[n]} $ on 
$\F^{(n)}$. Since 
$\F^{(n)} \cap \K(X^{\otimes n+1})
= \K(X^{\otimes n+1}) \cap \K(X^{\otimes n}) 
= \K(X^{\otimes n}J_X) \otimes I $, 
$\tau^{(b,r)}_{[n]}$ and $\tau^{(b,r)}_{n+1}$ are equal to zero on the 
intersection $\F^{(n)} \cap \K(X^{\otimes n+1})$. 
Therefore 
by Lemma \ref{lem:extension},
we can extend to  a finite trace $\tau^{(b,r)}_{[n+1]}$ on
$\F^{(n+1)}$.

Hence 
we have constructed finite traces
$\tau^{(b,r)}_{[n]}$ on $\F^{(n)}$ $n=0,1,2,\dots$  such that
$\left.\tau_{[n]}^{(b,r)}\right|_{\F^{(n')}} = \tau_{[n']}^{(b,r)}$
for $0 \le n' \le n$.
The family $\tau^{(b,r)}_{[n]}$ on $\F^{(n)}$ $n=0,1,2,\dots$ 
gives a finite trace $\tau^{(b,r)}$ on $\F^{(\infty)}$ such
$\tau^{(b,r)}|A = \mu^{(b,r)}_0$ and 
$\left.\tau^{(b,r)}\right|_{\K(X^{\otimes i})}=\pi_i(\mu^{(b,r)}_i)$ 
for  $0  \le i \le r$ and 
$\left.\tau^{(b,r)}\right|_{\K(X^{\otimes i})}= 0$ for  $i >  r$.

\end{proof}

We shall construct  a continuous type trace $\tau^{(\infty)}$ 
which corresponds
to the Hutchinson measure $\mu_H$ on $K$.

For $f \in {\rm C}(K)$, we put
\[
 G(f)(y) = \frac{1}{N} \sum_{j=1}^N f(\gamma_j(y)).
\]
Then $G:{\rm C}(K) \rightarrow {\rm C}(K)$ is a positive linear map. 
We note that
the dual map $G^*$ is a positive linear map on ${\rm C}(K)^*$.
We denote by $P(K)$ the space of probability measure on $K$.
We note that $G^*$ conserves $P(K)$. 
We shall collect some folklore facts with simle direct proofs: 

\begin{lemma} \label{lemma:contraction}
Let $\gamma$ be a self-similar map on $K$. 
Then we have the following: 
\begin{enumerate}
\item There exists a metric $L$ on a compact space $P(K)$ such
that  $G^*$ is a proper contraction with respect to $L$.
\item
There exists a unique fixed 
point $\mu_H$ in $P(K)$ under $G^*$ , 
which is called the Hutchinson measure $\mu_H$
on $K$, that is, $G^*(\mu_H) = \mu_H$. 
\item $\cap _{n=1}^{\infty} G^{*n}(P(K)) = \{\mu_H \}$
\item
 $\frac{1}{N}F(\mu_H) = G^*(\mu_H)= \mu_H$. 
\end{enumerate}

\end{lemma}
\begin{proof}(1)
We denote by ${\rm Lip}(K)$ the set of Lipschitz continuous functions
on $K$.  Recall that the Lipschitz semi-norm $\|\cdot \|_L$ on ${\rm
Lip}(K)$ is defined by
\[
 \|a\|_L = \sup_{x\ne y}\frac{|a(x)-a(y)|}{d(x,y)}, \ \ \text{ for } a \in C(K). 
\]
We define a function $L$ on $P(K) \times P(K)$ by
\[
 L(\mu,\nu) = \sup_{\|a\|_L \le 1}|\mu(a)-\nu(a)|.
\]
By Hutchinson \cite{H}, $L$ is a metric on $P(K)$  and the topology
on $P(K)$ given by $L$ coincides with the w*-topology.

We shall show
\[
 L(G^*(\mu),G^*(\nu)) \le c L(\mu,\nu)
\]
using $d(\gamma_j(x),\gamma_j(y)) \le c d(x,y)$ for each $j$.
We have
\begin{align*}
 \frac{|G(a)(x)-G(a)(y)|}{d(x,y)}
  \le & \frac{1}{N} \sum_{j=1}^N \frac{|a(\gamma_j(x)) -
                      a(\gamma_j(y))|}{d(x,y)} \\
  \le & \frac{1}{N} \sum_{j=1}^N \frac{|a(\gamma_j(x))-a(\gamma_j(y))|}
                  {d(\gamma_{j}(x),d(\gamma_j(y)))}
       \frac{d(\gamma_j(x)),\gamma_j(y))}{d(x,y)} \\
   \le & c\|a\|_L,
\end{align*}
and it follows that $\|G(a)\|_L \le c\|a\|_L$.  Then we have
\begin{align*}
 L(G^*(\mu),G^*(\nu)) = & \sup_{\|a\|_L \le 1} |\mu(G(a))-\nu(G(a))| \\
                 \le & \sup_{ \|a\|_L \le 1 } \|G(a)\|_L L(\mu,\nu) \\
                  \le & \sup_{\|a\|_L \le 1} c\|a\|_L L(\mu,\nu) \le
cL(\mu,\nu).  
\end{align*}
It holds that $ L(G^*(\mu),G^*(\mu)) \le cL(\mu,\nu)$.
This shows that $G^*$ is a proper contraction with respect to $L$.\\
(2) is a consequence of (1). \\
(3)Since $P(K)$ is compact, the diameter of $P(K)$ 
with respect to $L$ is finite. Therefore (1) and (2) implies (3). \\
(4)Since $\mu_H$ does not have point mass, for $a \in {\rm C}(K)$
\[ 
(\frac{1}{N}F(\mu_H))(a) 
= \frac{1}{N}\mu_H(\tilde{a})
=\mu_H(G(a))= 
(G^*(\mu_H))(a) = \mu_H(a).
\]
\end{proof}

\begin{prop}
Let $\gamma$ be a self-similar map on $K$ with assumption A. 
Let $\mu_H$ be the Hutchinson measure on $K$. Then 
there exists a unique 
finite trace $\tau^{\infty}$ on $\F^{(\infty)}$ such that 
$\tau^{\infty}(a) = \int a \mu_H$ for $a \in A = C(K)$ and 
\[
\tau^{\infty}(\theta_{x,y}) = \frac{1}{N^{i}}\mu_H( (y|x)_A) \ 
\text{ for } \ \ x,y \in X^{\otimes i} . 
\]
\end{prop}
\begin{proof}
Difine a trace  $\tau^{\infty}_i :=\pi_i(\mu_H)$ on 
$\K(X^{\otimes i})$ for $i = 0,1,2,\dots$, that is,  
$\tau^{\infty}(a) = \int a \mu_H$ for $a \in A = C(K)$ and 
\[
\tau^{\infty}(\theta_{x,y}) = \frac{1}{N^{i}}{\mu_H}( (y|x)_A)) \ 
\text{ for } \ \ x,y \in X^{\otimes i} . 
\]
We shall construct 
finite traces $\tau^{\infty}_{[n]}$ on $\F^{(n)}$ $n=0,1,2,\dots$ 
by induction on $n$.
By Corollarly \ref{cor:compact}, 
for $0 \leq i \leq n$ and $x,y \in X^{\otimes i}$,  we have  that
\[
 \overline{\tau^{\infty}_{n+1}}(\theta_{x,y} \otimes I) 
= (\left(\frac{1}{N^{i}}\right) (\left(\frac{1}{N^{n+1-i}}\right)   
  (F^{n+1-i}(\mu_H))((y|x)_A)
=(\left(\frac{1}{N^{i}}\right)(\mu_H)((y|x)_A) = 
\tau^{\infty}_{[n]}(\theta_{x,y} \otimes I)
\]
Thus 
$\overline{\tau^{\infty}_{n+1}}= \tau^{\infty}_{[n]} $ on 
$\F^{(n)}$.
Therefore 
by Lemma \ref{lem:extension},
we can extend to  a finite trace $\tau^{\infty}_{[n+1]}$ on
$\F^{(n+1)}$.
\end{proof}

\section
{Classification of traces on the core}

We investigate the place where  point masses appear on $K$ 
if a normalized trace
$\tau$
on $\F^{(\infty)}$ is restricted to $A$.

\begin{prop} \label{prop:no-pointmass} Let $\gamma$ be a self-similar map on $K$ with assumption A.
Then the restriction $\tau_0$ of a finite trace 
$\tau$ on $\F^{(\infty)}$
to $A$ does not have point mass except for points in 
$Orb = \cup_{b \in B_{\gamma}} \cup_{r=0}^{\infty} O_{b,r}$, 
where $O_{b,r} = h^{-r}(b)$. 
\end{prop}

\begin{proof}Let $\mu_{0} = \psi_0(\tau_{0})$ be the corresponding 
 Borel measures on $K$.  
Take a point  $a \in K$ which is not in $Orb$.  
Then for any $n = 0,1,2,3,\dots $, $a_n:=h^n(a) \notin B_{\gamma}$, 
where $a_0 = a$. Assume that $\mu_0(\{a\}) > 0. $ 
We shall show that $\#(h^{-n}(a_n)) \geq N^{n-1}$. Firstly,  consider 
the case that $h^{-1}(a_n)$ contains an element $b_{n-1} \in B_{\gamma}$. 
Then 
\[
(\cup_{r =1}^{n-1} h^{-r}(b_{n-1})) \cap B_{\gamma} = \emptyset 
\]
On the contrary suppose that 
$ (\cup_{r =1}^{n-1} h^{-r}(b_{n-1})) \cap B_{\gamma}$ 
was not empty. Then there exists an element  
$c \in (\cup_{r =1}^{n-1} h^{-r}(b_{n-1})) \cap B_{\gamma}$. 
Then there exists some $r= 1, \dots, n-1$ with $h^r(c) = b _{n-1}$. 
This contradicts to Assumption A \ (3). Therefore 
$(\cup_{r =1}^{n-1} h^{-r}(b_{n-1})) \cap B_{\gamma} = \emptyset $. 
Then $\#(h^{-(n-1)}(b_{n-1})) = N^{n-1}$. Since $\#(h^{-1}(a_{n}) \geq 1$, 
we have that $\#(h^{-n}(a_n)) \geq N^{n-1}$.

Secondly, consider 
the case that $h^{-1}(a_n) \cap B_{\gamma} = \emptyset$ and 
$h^{-2}(a_n)$ contains an element $b_{n-2} \in B_{\gamma}$. Then a 
similar argument shows that 
\[
(\cup_{r =1}^{n-2} h^{-r}(b_{n-2})) \cap B_{\gamma} = \emptyset 
\]
Thus $\#(h^{-(n-2)}(b_{n-2})) = N^{n-2}$ and $\#(h^{-1}(a_{n})) \geq N$. 
Therefore $\#(h^{-n}(a_n)) \geq N^{n-1}$.  
Consider the case that $h^{-i}(a_n) \cap B_{\gamma} = \emptyset$ 
for $i = 1,2, \dots, r-1$ and 
$h^{-r}(a_n)$ contains an element $b_{n-r} \in B_{\gamma}$ with 
$r \leq n$. 
A similar argument shows that $\#(h^{-n}(a_n)) \geq N^{n-1}$.
Finally consider the case that $h^{-i}(a_n) \cap B_{\gamma} = \emptyset$ 
for $i = 1,2, \dots, n$. Then $\#(h^{-n}(a_n)) \geq N^n$. 
In any cases, we have $\#(h^{-n}(a_n)) \geq N^{n-1}$.

By Lemma \ref{lem:point}, for any $x \in h^{-n}(a_n)$, 
\[
\mu_0(\{x\}) \geq \frac{\mu_n(\{a_n\})}{N^n} = \mu_0(\{a_0\})> 0. 
\]
Therefore $\mu_0(K) \geq \mu_0(h^{-n}(a_n)) \geq N^{n-1}\mu_0(\{a_0\})> 0.$ 
Taking $n \rightarrow \infty$, $\mu_0(K) = \infty$. 
This contradicts to that $\mu_0$ is a finite measure. 
Therefore  $\mu_0(\{a\}) = 0. $ 
\end{proof}

\begin{prop} Let $\gamma$ be a self-similar map on $K$ with assumption A.
Let  $\tau_0$ be the restriction of a finite trace 
$\tau$ on $\F^{(\infty)}$ to $A$. 
Let $\mu_{0} = \psi_0(\tau_{0})$  be the corresponding 
 Borel measures on $K$ 
by modified Riefell correspondence.
For $b \in B_{\gamma}$, $r \ge 0$, take any 
$x_{b,r}, x'_{b,r} \in O_{b,r}= h^{-r}(b)$, then 
point masses $\mu_0(\{x_{b,r}\}) = \mu_0(\{x'_{b,r}\})$
\end{prop}
\begin{proof}
This follows directly from Lemma \ref{lem:point}.
\end{proof}

Let $\tau$ be a finite trace on $\F^{(\infty)}$ 
and $\tau_i$ be the restriciton of $\tau$ to 
$\K(X^{\otimes i})$. 
Let $\mu_{i} = \psi_i(\tau_{i})$  be the corresponding 
 Borel measures on $K$ 
by modified Riefell correspondence.
For $b \in B_{\gamma}$ and $r \ge 0$, 
the proposition above helps us to  define a constant 
$c_{b,r} \ge 0$ by the point mass $c_{b,r}= \mu_0(\{x_{b,r}\})$ for 
any  $x_{b,r} \in O_{b,r}= h^{-r}(b)$. 
Let $0 \le i \le r$.  By Lemma \ref{lem:point}, the point mass of
$\mu_i$ at each point of 
$x \in O_{b,r-i}$ is determined by $\mu_i(\{x\}) = c_{b_,r}N^{i}$ 
and does not depend on the choice of such $x \in O_{b,r-i}$. 

\begin{lemma}
Let $\gamma$ be a self-similar map on $K$ with assumption A.
For any $b, b' \in B_{\gamma}$ and integers $r, r' \ge 0$, 
if $(b,r) \not= (b',r')$, then  $O_{b,r} \cap O_{b',r'}= \emptyset$. 
Moreover for any 
$z \in Orb = \cup_{b \in B_{\gamma}} \cup_{r=0}^{\infty} O_{b,r}$, 
there exist a unique $b \in B_{\gamma}$ and a unique 
integer $r \ge 0$ such that
$z \in O_{b,r}$. 
 
\end{lemma}
\begin{proof}
On the cotrary, assume that 
$O_{b,r} \cap O_{b',r'} \not= \emptyset$. Then there exists 
$x \in O_{b,r} \cap O_{b',r'}$.

Suppose that $b = b'$, then $r \not= r'$. We may assume that 
$r' > r$.  Since $h^r(x) = b = h^{r'}(x)$, 
we have $h^{r'-r}(b) = b$. This contradicts to Assumption A \ (3).

Suppose that $b \not= b'$. Since $h^r(x) = b$ and $h^{r'}(x) =b'$, 
$r \not= r'$. We may assume that $r' > r$. Then we have 
$h^{r'-r}(b) = h^{r'-r}(h^r(x)) = h^{r'}(x) =b'$. 
This contradicts to Assumption A \ (3).
Therefore $O_{b,r} \cap O_{b',r'}= \emptyset$. The  rest is clear. 
\end{proof}

Recall that we difine an endomorphism $\alpha$ on $C(K)$ by 
$(\alpha(a))(x) = a(h(x))$ for $a \in A = C(K)$ and $x \in K$. 
Then $(\alpha^i(a))(x) = a(h^i(x))$. 

\begin{prop}\label{prop:point-mass}
Let $\gamma$ be a self-similar map on $K$ with assumption A.
Let $\tau$ be  a finite trace $\tau$ on $\F^{(\infty)}$. 
For $b \in B_{\gamma}$ and integer $r \ge 0$, 
put a constant 
$c_{b,r} \ge 0$ by the point mass $c_{b,r}= \mu_0(\{x_{b,r}\})$ for 
any  $x_{b,r} \in O_{b,r}= h^{-r}(b)$. 
Then $\tau - \sum_{b \in B_{\gamma}}
\sum_{r=0}^{\infty} N^r c_{b,r} \tau^{(b,r)}$ is a positive finite trace
on $\F^{(\infty)}$ whose restriction to $A$ does not have positive
point mass.
\end{prop}
\begin{proof}
We shall show  that 
$\tau \ge 
\sum_{b \in B_{\gamma}}\sum_{r=0}^{\infty} N^r c_{b,r} \tau^{(b,r)}$. 
Since $c_{b,r}= \mu_0(\{x_{b,r}\})$ for 
any  $x_{b,r} \in O_{b,r}= h^{-r}(b)$, 
we note that 
\[
\sum_{b \in B_{\gamma}}\sum_{r=0}^{\infty} N^r c_{b,r}
= \mu_0(\cup_{b \in B_{\gamma}}\cup_{r=0}^{\infty} O_{b,r}) \leq 1.
\]
It is enough to show that for any finite positive integer $s$.
\[
\tau \ge \sum_{b \in B_{\gamma}}\sum_{r=0}^{s}
N^r c_{b,r} \tau^{(b,r)}
\]

Since $\tau$ and $\sum_{b \in B_{\gamma}}\sum_{r=0}^s\tau^{(b,r)}$ are
finite traces, it is sufficient to
show that
\[
 \tau(T) \ge \sum_{b \in B_{\gamma}}\sum_{r=0}^{s}
 N^r c_{b,r} \tau^{(b,r)}(T)
\]
for $T \in {\F^{(p)}}^+$ for each
positive integer $p$ with $s+1 \le p$.
We fix a positive integer $p$ and 
integers $r \ge 0$ such that $r+1 \le p$.
For $b\in B_{\gamma}$ and integers 
$i$ with $r \le i \le p$, we can choose a sequences
$\{g^{(b,r,i)}_m\}_{m=1,2,\dots}$ of elements in $A= C(K)$ satisfying the
following:
\begin{enumerate} 
 \item $0 \le g_m^{(b,r,i)}(x) \le 1$.
 \item $g_m^{(b,r,i)}(h^{i-r}(b))=1$.
 \item For each $i$, the diameter of the support of $g_m^{(b,r,i)}$
approaches
       zero as $m$ tends to $\infty$.
 \item $\alpha_r(g^{(b,r,r)}_m)(x) \le \alpha_i(g^{(b,r,i)}_m)(x)$ for
       each $x \in K$.
 \item For any $m$, if $(b,r) \not= (b',r')$,then 
the supports of $\alpha_r(g^{(b,r,r)}_m)$ and 
$\alpha_{r'}(g^{(b',r',r')}_m)$ are disjoint. 
\end{enumerate}
In fact, for $b \in B_{\gamma}$ and an integer $r \ge 0$, 
$O_{b,r} = h^{-r}(b)$ is a finite set.  Moreover 
if $(b,r) \not= (b',r')$, then  $O_{b,r} \cap O_{b',r'}= \emptyset$. 
Hence there exist open neibourhoods $U_{b,r}$  for $b \in B_{\gamma}$ and 
$0 \leq r \leq p$
such that $h^{-r}(U_{b,r}) \cap h^{-r'}(U_{b',r'}) = \emptyset$  if $(b,r) \not= (b',r')$. 
For $b\in B_{\gamma}$ and integers 
$i$ with $r <i \le p$, it is easy to find  a sequences
$\{g^{(b,r,i)}_m\}_{m=1,2,\dots}$  and 
$\{g'^{(b,r,r)}_m\}_{m=1,2,\dots}$ of 
 elements in $A= C(K)$ satisfying the conditions (1) -(3) above 
 and the support of $g^ {(b,r,r)}_m$ is included in $U_{b,r}$. 
 Then define 
 \[
g^{(b,r,r)}_m (y) = g'^{(b,r,r)}_m(y) \prod _{i=r+1}^p g^{(b,r,i)}_m (h^{i-r}y)
 \]
 
We put $f^{(b,r)}_m=\alpha_r(g^{(b,r,r)}_m)$.
Let $Q = \sum_{b \in B_{\gamma}}\sum_{r=0}^sf^{(b,r)}_m$. 
Since $0 \le Q \le 1$,
we have
\begin{align*}
\tau(T) \geq \tau (T^{1/2}QT^{1/2}) = \tau(TQ) 
 = \sum_{b \in B_{\gamma}} \sum_{r=0}^{s} \tau(T f^{(b,r)}_m)
\end{align*}
for $T \in \F^{(p)}{}^+$.
We shall show that
\[
 \lim_{m \to \infty}\tau_i(Tf^{(b,r)}_m) = N^rc_{b,r}\tau^{(b,r)}_i(T)
\]
for $T \in \K(X^{\otimes i})$ with $0 \le i \le p$.

Since $\tau_i$ and $\tau^{(b,r)}_i$ are finite trace and $0 \le
f_m^{(b,r)} \le 1$,
it suffices to prove Lemma for a dense subset of $\K(X^{\otimes i})$.
We may assume that $T = \theta_{\xi,\eta}$ for 
$\xi$, $\eta \in X^{\otimes i}$.

Firstly, we consider the case that  $0 \le i \le r$.  
Let $\xi$, $\eta \in X^{\otimes i}$.
Then we have
 \begin{align*}
 \lim_{m \to \infty}
\tau_i(\theta_{\xi,\eta}\phi_i(\alpha^r(g^{(b,r,r)}_m)))
 = & \lim_{m \to
\infty}\tau_i(\theta_{\xi,\eta}\alpha^i(\alpha^{r-i}(g^{(b,r,r)}_m))) \\
 = & \lim_{m \to \infty}
 \tau_i\left(\theta_{\xi,\phi_i(\alpha^i(\alpha^{r-i}(g^{(b,r,r)}_m)))\eta}\right) \\
 = & \lim_{m \to \infty}
   N^i \psi_i(\tau_i)(\phi_i(\alpha^i(\alpha^{r-i}(g^{(b,r,r)}_m)))\eta |
\xi)_A) \\
 = & \lim_{m \to \infty}
    N^i\psi_i(\tau_i) ((\eta \alpha^{r-i}(g^{(b,r,r)}_m) | \xi)_A) \\
 = & \lim_{m \to \infty}
    N^i\psi_i(\tau_i)((\alpha^{r-i}(g^{(b,r,r)}_m) (\eta|\xi)_A) \\
 = & \sum_{(j_1,\dots,j_{r-i}) \in \Sigma^{r-i}}
     N^i c_{b,r} (\eta|\xi)_A(\gamma_{j_1}\circ \cdots \circ j_{r-i}(b)) \\
 = & c_{b,r}N^r (\tau^{(b,r)}_i)(\theta_{\xi,\eta}).
\end{align*}
Hence 
\[
 \lim_{m \to \infty} \tau_i(Tf^{(b,r)}_m) = c_{b,r}N^r \tau^{(b,r)}_i(T)
\]
for $T \in \K(X^{\otimes i})$ with $0 \le i \le r$.

Secondly, we consider the case  that $r+1 \le i \leq p$. 
 Let $\xi$, $\eta \in X^{\otimes i}$.
We have
\begin{align*}
  \lim_{m \to \infty}
\tau_i\left(\theta_{\xi,\eta} \phi_i(\alpha^i(g^{(b,r,i)}_m))\right)
 =  &   \lim_{m \to \infty}
   \tau_i\left(\theta_{\xi,\phi_i(\alpha^i(g^{(b,r,i)}_m)) \eta} \right)  \\
 = & \lim_{m \to \infty} \psi_i(\tau_i)
   \left( \phi_i(\alpha^i(g^{(b,r,i)}_m))\eta | \xi)_A \right)\\
 = & \lim_{m \to \infty} \psi_i(\tau_i)
   \left( (\eta g^{(b,r,i)}_m| \xi)_A \right)\\
= & \lim_{m \to \infty} \psi_i(\tau_i)
   \left( g^{(b,r,i)}_m (\eta | \xi)_A \right) = 0,
\end{align*}
because $\psi_i(\tau_i)$ does not have a positive point mass at
$h^{i-r}(b) \in C_{\gamma} \cup P_{\gamma}$, 
 by Propsition \ref{prop:no-pointmass}. Then we have
\[
 \lim_{m \to \infty}\tau_i(T \alpha^i(g^{(b,r,i)}_m)) = 0
\]
for $T \in \K(X^{\otimes i})$.  We assume $T \in \K(X^{\otimes i})$ be
positive.  Then we have
\begin{align*}
\tau_i(Tf^{(b,r)}_m)
 =& \tau_i(T^{1/2}\alpha^r(g^{(b,r,r)}_m)T^{1/2}) \\
 \le & \tau_i(T^{1/2}\alpha^i(g^{(b,r,i)}_m)T^{1/2}) \\
 =  & \tau_i(T \alpha^i(g^{(b,r,i)}_m)).
\end{align*}
Since $\tau_i(T \alpha^i(g^{(b,r,i)}_m))$ approaches $0$ as $m$ tends to
$\infty$,
we have
\[
 \lim_{m \to \infty}\tau_i(Tf^{(b,r)}_m)=0.
\]
for $T \ge 0$.  Since each $T \in \K(X^{\otimes i})$ is expressed as
$T = T_1 - T_2 + i(T_3 -T_4)$ with  $T_i \ge 0$, we have
\[
 \lim_{m \to \infty}\tau_i(Tf^{(b,r)}_m)=0 = c_{b,r}N^r \tau^{(b,r)}_i(T)
\]
for $T \in \K(X^{\otimes i})$.
\end{proof}

Nextly we shall analyze traces without point mass. 

\begin{lemma} Let $\tau$ be a finite trace on $\F^{(\infty)}$.
If the restriction $\tau_0 $of $\tau$ to $A$ does not have positive point mass,
then $\tau_n = \psi_{n}(\tau|_{\K(X^{\otimes n})})$ also does not have positive
point mass.
\end{lemma}
\begin{proof}
If $\psi_{n}(\tau|_{\K(X^{\otimes n})}))$ 
has a positive point mass,then $\tau_0$ also has a positive point mass at
some point by Lemma \ref{lem:point}. 
\end{proof}

It is a key point that we  can replace $F$ by $G^*$.

\begin{lemma} Let $\sigma$ be a finite trace of $\F^{(\infty)}$ 
and $\sigma_n = \sigma|_{\K(X^{\otimes n})}$ the restriction of $\sigma$ 
to $K(X^{\otimes n})$. 
If the corresponding measure 
$\psi_n(\sigma_n)$ does not have positive point
mass for any $n$, then it holds that for any $n$ 
\[
\frac{1}{N}F(\psi_{n+1}(\sigma_{n+1})) = G^{*}(\psi_{n+1}(\sigma_{n+1})).
\]
Moreover  $\mu_n = \psi_n(\sigma_n)$ is a provability
measure and  $\mu_n = G^*(\mu_{n+1})$ for  $n$.
\end{lemma}
\begin{proof} Since $\psi_{n+1}(\sigma_{n+1})$ has no positive point mass,
it holds that
\[
\frac{1}{N}F(\psi_{n+1}(\sigma_{n+1}))(a)
= \frac{1}{N}\psi_{n+1}(\sigma_{n+1})(\tilde{a})
= \psi_{n+1}(\sigma_{n+1})(G(a))
= G^*(\psi_{n+1}(\sigma_{n+1}))(a).
\]
It holds that $\sigma_0(I) = \mu_0(1_{K}) = 1$, and $\mu_0$ is a
probability measure.  By
\[
 \mu_0(1_{K}) = G^*(\mu_1)(1_{K}) = \mu_1(G(1_{K}))
 = \mu_1(1_{K}),
\]
$\mu_1$ is also a probability measure.  Inductively, $\mu_n$ is
a probability measure for each $n$.

\end{proof}

We denote by $\mu_{H}$ the Hutchinson measure on $K$.  It is shown 
that
there exists a unique tracial state 
$\tau^{(\infty)}$ on $\F^{(\infty)}$ such that
$\tau^{(\infty)}_n =\mu_{H}$ for each $n$.  The tracial state
$\tau^{(\infty)}$
gives a continuous type (infinite type) KMS state on $\O_{\gamma}$.

\begin{prop} \label{prop:Hutchinson} Let $\gamma$ be a self-similar map on $K$ with assumption A.
Let $\sigma$ be a tracial state on $\F^{(\infty)}$.  If 
the restriction $\mu_0$ of $\sigma$ to $A$  
does not have positive point mass, then $\sigma$ coincides with
$\tau^{(\infty)}$.
\end{prop}
\begin{proof}Let 
$\sigma_n = \sigma|_{\K(X^{\otimes n})}$ be the restriction of $\sigma$ 
to $K(X^{\otimes n})$ and $\mu_n = \psi_n(\sigma_n)$
be the corresponding measure.  
Since 
$\frac{1}{N}F(\psi_{n+1}(\sigma_{n+1})) = G^{*}(\psi_{n+1}(\sigma_{n+1}))$ 
and $\frac{1}{N}F(\psi_{n+1}(\sigma_{n+1})) = \psi_{n}(\sigma_{n})$, 
we have 
$G^*(\mu_{n+1})=\mu_n$ ($n=0,1,\dots$). Hence
\[
\mu_n \in \bigcap_{k = n+1}^{\infty} G^{*k}(P(K)) 
= \bigcap_{k = 1}^{\infty} G^{*k}(P(K))
\]
Therefore  $\mu_n = \mu_H$ for any $n$. This implies that 
$\sigma  = \tau^{(\infty)}$
\end{proof}

We state the following Theorem on the classification of extreme 
tracial states on the core of $C^*$-algebras constructed from
self-similar maps.

\begin{thm}  
Let $\gamma$ be a self-similar map on $K$ with assumption A and 
$\O_{\gamma}$ the $C^*$-algebra associated with $\gamma$. 
Then for any 
finite trace $\tau$ on $\F^{(\infty)}$, 
there exist  unique constants $c_{b,r}\geq 0$ and 
$c_{\infty} \geq  $ such that 
\[
\tau = \sum_{b \in B_{\gamma}}
\sum_{r=0}^{\infty} N^r c_{b,r} \tau^{(b,r)} + c_{\infty}\tau^{\infty}
\]
Moreover, 
the  set of extreme tracial states of the core
$\F^{(\infty)}$ of $\O_{\gamma}$ is exactly 
\[
\{\tau^{(\infty)}\} \bigcup
\{\tau^{(b,r)}\,|\,b \in B_{\gamma},\, r=0,1,2,\dots \}.
\]

\end{thm}
\begin{proof}
Let $\tau$ be  a finite trace $\tau$ on $\F^{(\infty)}$. 
Define $c_{b,r}= \mu_0(\{x_{b,r}\})$ for 
any  $x_{b,r} \in O_{b,r}= h^{-r}(b)$. 
Then $\sigma := \tau - \sum_{b \in B_{\gamma}}
\sum_{r=0}^{\infty} N^r c_{b,r} \tau^{(b,r)}$ is a positive finite trace
on $\F^{(\infty)}$ whose restriction to $A$ does not have positive
point mass by Proposition \ref{prop:point-mass}.

If $\sigma = 0$, then nothing is to be proved.

If 
$\sigma \not= 0$, put $c_{\infty} = \sigma (1) \not= 0$. Then 
$\frac{1}{c_{\infty}}\sigma$ is a tracial state on $\F^{(\infty)}$ and 
its restriction  to $A$  
does not have positive point mass. 
Therefore $\frac{1}{c_{\infty}}\sigma$ coincides with
$\tau^{(\infty)}$ by Proposition \ref{prop:Hutchinson}. 
 Hence we have 
\[
\tau = \sum_{b \in B_{\gamma}}
\sum_{r=0}^{\infty} N^r c_{b,r} \tau^{(b,r)} + c_{\infty}\tau^{\infty}
\]
We shall show the uniqueness of the expression. Assume that
there exist   constants $d_{b,r}\geq 0$ and 
$d_{\infty} \geq  $ such that 
\[
\tau = \sum_{b \in B_{\gamma}}
\sum_{r=0}^{\infty} N^r d_{b,r} \tau^{(b,r)} + d_{\infty}\tau^{\infty}
\]
Then $d_{b,r}= \mu_0(\{x_{b,r}\})= c_{b,r}$, because we know 
the point mass of the restrictions of $\tau^{(b,r)}$ to $A=C(K)$ and 
the point mass of the restriction of $\tau^{\infty}$ to $A=C(K)$ is 
zero. Then $c_{\infty}\tau^{\infty} = d_{\infty}\tau^{\infty}$. 
This implies $c_{\infty} = d_{\infty}$. Hence the expression is unique.

The uniqueness of the expression implies that 
the  set of extreme tracial states of the core
$\F^{(\infty)}$ of $\O_{\gamma}$ is exactly 
\[
\{\tau^{(\infty)}\} \bigcup
\{\tau^{(b,r)}\,|\,b \in B_{\gamma},\, r=0,1,2,\dots \}.
\]

\end{proof}

We explain the construction of extreme traces by two typical examples.

\begin{exam} We describe the  extreme traces on the core of the
 $C^*$-algebras associated with the tent map explicitly.
We define discrete measures $\mu_i^{(r)}$on $[0,1]$ for $0 \le i \le r$ by
\[
 \mu_i^{(r)}(f) = \frac{1}{2^{r-i}}\sum_{(j_1,\dots,j_{r-i}) \in
\{\,1,2\,\}^i}
    f(\gamma_{j_1}\circ \cdots \circ \gamma_{j_{r-i}}(\delta_{1/2})).
\]
We define traces $\tau^{(r)}$ on $\F^{(\infty)}$ for $r \ge 0$ by 
Morita equivalence and 
 \[
  \tau^{(r)}_i =
\begin{cases}
 & \pi_i(\mu^{(r)}_i) \quad 0 \le i \le r \\
 & 0 \quad r+1 \le i
\end{cases}
 \]
for each $i$.
We can also define a trace $\tau^{(\infty)}$ on $\F^{(\infty)}$ by
\[
 \tau^{(\infty)}_i = \mu_{\infty},
\]
for each $i$ where $\mu_{\infty} = \mu_H $ is the normalized Lebesgue measure on
$[0,1]$ and coincides with  the Hutchinson measure.  
The the set of extreme traces on the core of $C^*$-algebras 
associated with the tent map on $[0,1]$ is exactly 
 $\{\,\tau^{(0)},\tau^{(1)},\dots,\tau^{(\infty)}\,\}$.
\end{exam}

\begin{exam}
We write down extreme traces on the core of the $C^*$-algebras
associated the self-similar map generating the Sierpinski gasket $K$ .
We define measures on $K$ by
\begin{align*}
\mu^{(S,r)}_i  = & \frac{1}{3^{r-i}} \sum_{(j_1,\dots,j_{r-i}) \in
 \{\,1,2,3\,\}^{r-i}}f(\gamma_{j_1}\circ \cdots \circ
\gamma_{j_{r-i}}(\delta_S))  \\
\mu^{(T,r)}_i  = & \frac{1}{3^{r-i}} \sum_{(j_1,\dots,j_{r-i}) \in
 \{\,1,2,3\,\}^{r-i}}f(\gamma_{j_1}\circ \cdots \circ
\gamma_{j_{r-i}}(\delta_T))\\
\mu^{(U,r)}_i  = & \frac{1}{3^{r-i}} \sum_{(j_1,\dots,j_{r-i}) \in
 \{\,1,2,3\,\}^{r-i}}f(\gamma_{j_1}\circ \cdots \circ
\gamma_{j_{r-i}}(\delta_U))
\end{align*}
Then we define traces $\tau^{(S,r)}$, $\tau^{(T,r)}$ and $\tau^{(U,r)}$
for $r=0,1,2,\dots$ using $\mu^{(S,r)}_i$s, $\mu^{(T,r)}_i$s and
 $\mu^{(U,r)}_i$s by Morita equivalence. 
We can also define a trace $\tau^{(\infty)}$ on $\F^{(\infty)}$ by
\[
 \tau^{(\infty)}_i = \mu_{H},
\]
for each $i$ where $\mu_{H}$ is the Hutchinson measure on the
Sierpinski
Gasket $K$. 
\end{exam}

Finally, we describe a  connection of the extreme
traces of the core and the KMS
states of $\O_{\gamma}$. A general characterization of 
KMS states 
is given by Laca-Neshveyev \cite{LN}. 
Let $\beta > \log N$ and we put
\[
  \rho^{(b,\beta)}= C_{b,\beta} \sum_{j=0}^{\infty}
 \left(\frac{N}{e^{\beta}}\right)^{j}\tau^{(b,j)},
\]
where $C_{b,\beta}$ is a normalizing constant.
\begin{prop}
Let $\gamma$ be a self-similar map on $K$ with assumption A.
Then $\tau^{(\infty)}$ extends to a $\log N$-KMS state on $\O_{\gamma}$ and 
for $\beta >\log N$, $\rho^{(b,\beta)}$ extends to a  
$\beta$-KMS state
of $\O_{\gamma}$.
\end{prop}
\begin{proof}
This follows from the fact that the $\beta$-KMS state of $\O_{\gamma}$ 
is described concretely in section 6 of our paper \cite{IKW} with Izumi. 

\end{proof}

\end{document}